\documentclass[11pt]{amsart}
\usepackage{geometry}
\usepackage{graphicx}
\usepackage{amssymb}
\usepackage{amsmath, mathrsfs, stmaryrd}
\usepackage{color}
\newtheorem{theorem}{Theorem}
\newtheorem{definition}[theorem]{Definition}
\newtheorem{lemma}[theorem]{Lemma}
\newtheorem{corollary}[theorem]{Corollary}
\newtheorem{proposition}[theorem]{Proposition}
\newtheorem{remark}[theorem]{Remark}

\newcommand{\R}{\mathbb R}
\newcommand{\h}{\mathbb H}
\numberwithin{equation}{section}

\numberwithin{theorem}{section}
\begin{document}
\thanks{Part of this work was carried out while all three authors were visiting Department of Mathematics of National Taiwan University and Taida Institute for Mathematical Sciences in Taipei, Taiwan. P.-N. Chen is supported by NSF grant DMS-1308164, M.-T. Wang is supported by NSF grants DMS-1105483 and DMS-1405152 and S.-T. Yau is supported by NSF
grant  PHY-0714648.} 
\title[Conserved quantities]{Conserved quantities on  asymptotically hyperbolic initial data sets}
\author{Po-Ning Chen, Mu-Tao Wang and Shing-Tung Yau}
\begin{abstract}
In this article, we consider the limit of quasi-local conserved quantities \cite{Wang-Yau2, Chen-Wang-Yau3} at the infinity of an asymptotically hyperbolic initial data set in general relativity.
These give notions of total energy-momentum, angular momentum, and center of mass. Our assumption on the asymptotics is less stringent than any previous ones to validate a Bondi-type mass loss formula. The Lorentz group acts on the asymptotic infinity through the exchange of foliations by coordinate spheres. For foliations aligning with the total energy-momentum vector, we prove that the limits of  quasi-local center of mass and angular momentum are finite, and evaluate the limits in terms of the expansion coefficients of the metric and the second fundamental form.
 \end{abstract}
\maketitle

\section{Introduction}
The notion of energy, linear momentum, angular momentum, center of mass and their conservation laws are of fundamental importance for any physical theory. However, there have been great difficulties in understanding these notions for gravitation since Einstein's time, as there there is no well-defined concept of energy density due to the equivalence principle. It is nevertheless possible to use asymptotic symmetries to define these notions for an isolated system. At spatial infinity, the Arnowitt-Deser-Misner (ADM) energy-momentum \cite{Arnowitt-Deser-Misner} is well-understood. ADM energy-momentum is fundamental in general relativity and has been proven to be natural and to satisfy the important positivity property by the work of Schoen-Yau \cite{Schoen-Yau} and Witten \cite{Witten}. It is also shown that the ADM energy-momentum satisfies important invariant properties in the work of Bartnik \cite{Bartnik} and Chru\'{s}ciel \cite{Chrusciel}.  There are several existing definitions of total angular momentum and total center of mass such as the Arnowitt-Deser-Misner (ADM) angular momentum \cite{Arnowitt-Deser-Misner,Ashtekar-Hansen,Regge-Teitelboim} and the center of mass proposed by Huisken-Yau, Regge-Teitelboim, Beig-\'OMurchadha, Christodoulou and Schoen \cite{Huisken-Yau, Regge-Teitelboim, Beig-Omurchadha, Christodoulou, Huang}. In \cite{Chen-Wang-Yau3}, the authors proposed a new definition of quasi-local angular momentum and 
 center of mass and used their limits to define new total conserved quantities for asymptotically flat initial data sets. The new definitions are finite for asymptotically flat initial data of order 1, and satisfy important dynamical formulas for solutions of the Einstein equation. See Theorem 7.4 and Theorem 9.6 of \cite{Chen-Wang-Yau3}.

When the system is viewed from null infinity, the situation is more complicated. The notion of mass at null infinity is first studied by Bondi   \cite{Bondi-Burg-Metzner}  and Trautman \cite{Trautman}. While the ADM energy-momentum at spatial infinity is conserved for solutions of the Einstein equation, the Bondi energy at null infinity is decreasing, see \cite{Bondi-Burg-Metzner}, \cite{Sachs} and \cite{Trautman}. A similar mass loss formula at null infinity is derived in \cite{sta} by Christodoulou and Klainerman, as a consequence of their proof of global stability of the Minkowski space. Their mass loss formula plays a key role in the study of nonlinear memory effect in gravitational radiation \cite{Christodoulou2}.  In \cite{Rizzi}, Rizzi proposed a new definition of total angular momentum at null infinity by assuming the existence of a special foliation.  He also studied the change of this total angular momentum  by choosing a particular laspe function.

Schoen and Yau modified their proof for the positivity of ADM energy to prove the positivity of Bondi mass \cite{Schoen-Yau2}. Their main strategy is to study a spacelike hypersurface asymptotic to the null cone at infinity. Both the induced metric and the second fundamental form are asymptotic to the metric of the hyperbolic 3-space. On such an asymptotically hyperbolic hypersurface, they solve Jang's equation to obtain an asymptotically flat manifold whose ADM energy is a positive multiple of the Bondi mass of the null cone. 

Motivated by the proof of the positivity of Bondi mass and the study of asymptotically Anti-de-Sitter space, it is natural to find a suitable notion of mass and other conserved quantities for general asymptotically hyperbolic manifolds. In \cite{Wang}, X. Wang proved the positivity of mass for umbilical and conformally compact asymptotically hyperbolic manifolds satisfying the dominant energy condition. There are many works along this direction,  see for example \cite{Andersson-Cai-Galloway}, \cite{Chrusciel-Herzlich}, \cite{Chrusciel-Jezuerski}, \cite{Sakovich} and \cite{Zhang}. In \cite{Zhang2}, Zhang proved a positivity and rigidity theorem for the mass of asymptotically hyperbolic manifolds without the umbilical assumption. In the same article, a new definition of total angular momentum was also proposed. 

In this article, we study the total energy, linear momentum, angular momentum and center of mass at null infinity of asymptotically flat spacetime using the quasi-local energy of \cite{Wang-Yau2} and the quasi-local angular momentum and center of mass of \cite{Chen-Wang-Yau3}. The null infinity is modeled on 3-manifolds asymptotic to a hyperboloid in the Minkowski spacetime where the induced metric is isometric to the hyperbolic 3-space and the second fundamental form is the same as the induced metric. Let $\mathbb{H}^3$ denote the hyperbolic 3-space with metric $\frac{1}{r^2+1} dr^2+r^2 \tilde{\sigma}_{ab} du^a du^b$ where $\tilde \sigma_{ab}$ is the standard metric on the unit 2-sphere $S^2$. Specifically, here are the decay conditions:
\begin{definition}\label{a_h_coordinates}
A triple $(M,g,k)$ of a complete 3-manifold $M$, a Riemannian metric $g$ on $M$, and a symmetric 2-tensor (the second fundamental form) $k$ is said to be an  asymptotically hyperbolic initial data set if 
there exists a compact subset $K$ of $M$ such that $M\backslash K$ is diffeomorphic to a finite union  of ends $\cup_i  \mathbb{H}^3\backslash B_i$ where  each $B_i$ is a geodesic ball in $\mathbb{H}^3$. Under the diffeomorphism, we have
\[g= g_{rr} dr^2+ 2 g_{ra}dr du^a + g_{ab} du^adu^b \text{ and } k=g+p,\]
where
\[\begin{split}
g_{rr} =& \frac{1}{r^2} - \frac{1}{r^4}+ \frac{g_{rr}^{(-5)}}{r^5} + \frac{g_{rr}^{(-6)}}{r^6} +O(r^{-7}), \qquad  g_{ra}= \frac{g_{ra}^{(-3)}}{r^3}+O(r^{-4}), \\
g_{ab} =& r^2 \tilde \sigma_{ab} + g_{ab}^{(0)}+ \frac{g_{ab}^{(-1)}}{r}+ \frac{g_{ab}^{(-2)}}{r^2} +O(r^{-3}),
\end{split}\]
and
\[\begin{split}
p_{rr} =& \frac{p_{rr}^{(-4)}}{r^4}+O(r^{-5}) \qquad  p_{ra}=  \frac{p_{ra}^{(-3)}}{r^3}+O(r^{-4}),  \\
p_{ab} = & p^{(0)}_{ab}+  \frac{p^{(-1)}_{ab}}{r} + \frac{p_{ab}^{(-2)}}{r^2} +O(r^{-3}).
\end{split}\]
Here $\tilde \sigma_{ab}$ is the standard round metric on the unit 2-sphere $S^2$. 
$g_{rr}^{(-5)}, g_{rr}^{(-6)}, p_{rr}^{(-4)}$ are considered to be functions on $S^2$ that do not depend on $r$, $g_{ra}^{(-3)}, p_{ra}^{(-3)}$ are considered to be one-forms on $S^2$ that do not depend on $r$, and $g_{ab}^{(0)}, g_{ab}^{(-1)}, g_{ab}^{(-2)}, p_{ab}^{(0)}, p_{ab}^{(-1)}, p_{ab}^{(-2)}$
are considered to be symmetric two-tensors on $S^2$ that do not depend on $r$. Furthermore, we assume that  $p^{(0)}_{ab}$ and $g^{(0)}_{ab}$ are traceless with respect to $\tilde \sigma_{ab}$.
\end{definition}
We assume that the triple $(M,g,k)$ satisfies the following dominant energy condition:
\begin{definition}
$(M,g,k)$ satisfies the dominant energy condition if 
\[  
\begin{split}
\mu= & \frac{1}{2}(R(g)+ (tr_g k)^2- |k|_g^2) \\
{\mathfrak J}_i = &D^j(k_{ij} - (tr_g k )g_{ij})
\end{split}
\] 
satisfies
\[ \mu \ge |{\mathfrak J}|.  \]
Here $R(g)$ is the scalar curvature of the metric $g$ and $D^j$ is the covariant derivative with respect to the metric $g$.
\end{definition}

\begin{definition} \label{aspect}
The mass aspect function $m$ of an asymptotically hyperbolic initial data set is defined to be
\begin{equation}\label{mass_aspect} m= \frac{3}{2}tr_{S^2} g^{(-1)}_{ab}+  tr_{S^2}p^{(-1)}_{ab} +  g^{(-5)}_{rr}.   \end{equation}
\end{definition}
The energy-momentum of an asymptotically hyperbolic initial data set is defined as follows:
\begin{definition} \label{definition}
Let $(M,g,k)$ be an  asymptotically hyperbolic initial data set. The energy-momentum of $(M,g,k)$  is the four vector $(E_{AH},P_{AH}^i), i=1, 2, 3$ where
\[  
\begin{split}
E_{AH}(M,g,k)=& \frac{1}{8 \pi} \int_{S^2} m  \, dS^2\\
P_{AH}^i(M,g,k)= & \frac{1}{8 \pi} \int_{S^2} \tilde X^i m \, dS^2, i=1, 2, 3
\end{split}
\]
where $\tilde X^i, i=1, 2, 3$ are the three standard coordinate functions on $S^2$. 
\end{definition}
\begin{remark}
There are two types of asymptotically hyperbolic initial data sets. One is modeled on the hyperbola in the  Minkowski spacetime, and thus $g=k$ (umbilical) is the same as the metric on $\mathbb{H}^3$. The other is modeled on a static slice of the Anti-de-Sitter spacetime. Most studies only consider
a Riemannian manifold that is asymptotically to a hyperbolic space and it is implicitly assumed that the second fundamental form is either zero or the same as the induced metric. However, in studying the mass of an initial data set, it is important to take into account of the second fundamental form. For example, there exists an asymptotically umbilical spacelike hypersurface in the Schwarzschild spacetime with exactly hyperbolic induced metric and an asymptotically totally geodesic spacelike
hypersurface in the Anti-de-Sitter Schwarzschild spacetime with exactly hyperbolic induced metric. In both cases, the mass can only be read off from the 
second fundamental form. 
\end{remark}
\begin{remark} In \cite{Wang}, X. Wang defined energy-momentum for asymptotically hyperbolic Riemannian manifolds with $g_{ra}=0$ and $g_{rr}= \frac{1}{r^2+1}$. It is easy to see that his definition agrees with ours in this simplest case. 
\end{remark}

In \cite{Chen-Wang-Yau1} and \cite{Wang-Yau3}, the authors proved that the limit of the Wang-Yau quasi-local energy of the coordinates spheres recovers the ADM energy momentum vector at spatial infinity and the Bondi-Sachs energy momentum vector at null infinity. In setion 3, we prove

 \vskip 10pt
\noindent {\bf Theorem A} (Theorem \ref{limit_vector})
{\it
Given an asymptotically hyperbolic initial data set $(M,g,k)$, the limit of the quasi-local energy of $\Sigma_r$ with reference embedding $X_r$ into $\R^3$ is a linear function dual to the energy-momentum of $(M,g,k)$.
}
\vskip 10pt
In \cite{Kwong-Tam}, Kwong and Tam evaluated the limit of a quasi-local energy-momentum defined by Wang and Yau in \cite{Wang-Yau1} at the infinity of an asymptotically hyperbolic manifold. The quasi-local energy-momentum  of \cite{Wang-Yau1} again only uses the hyperbolic space as reference and does not include the effect of the second fundamental form of the initial data.

 In Section 4, we prove the following theorem regarding the causality of the energy-momentum vector of an asymptotically hyperbolic initial data set.
 \vskip 10pt
\noindent {\bf Theorem B} (Theorem \ref{thm_positivity})
{\it
Suppose  $(M,g,k)$ is an asymptotically hyperbolic initial data set satisfying the dominant energy condition, then its  energy-momentum 4-vector is future directed non-spacelike.  Namely,
\[   E_{AH}  \ge \sqrt{\sum_i (P^i_{AH})^2 } . \]
}
\vskip 10pt

To show that the energy-momentum 4-vector is non-spacelike, we apply a Lorentz boost of the Minkowski spacetime to the coordinate spheres, obtain
a new foliation, and compute the limit of quasi-local energy with respect to the new foliation. For any future directed unit timelike vector $(a^0,a^i)$ in $\R^{3,1}$,  we show that the limit of the Liu-Yau quasi-local energy of corresponding boost  is $a^0 E_{AH}- \sum_i a^i P_{AH}^i$. The positivity of the Liu-Yau quasi-local energy \cite{Liu-Yau} implies the 
energy momentum 4-vector is non-spacelike.

In addition to the positive mass theorem, we derive an energy loss formula for asymptotically hyperbolic initial data sets. Our  energy loss formula is similar to the Bondi mass loss formula derived in \cite{sta}. In particular, the source of the energy loss is $p_{ab}^{(0)}+ g_{ab}^{(0)}$, which plays the role of the leading order term of $\underline{\hat  \chi}$ in the mass loss formula derived in \cite{sta}. Theorem \ref{limit_vector} enables us to use the limit of the Liu-Yau mass for the total energy of asymptotically hyperbolic initial data set and to compute the variation of the Liu-Yau mass along the Einstein equation. This approach is similar to that of \cite{sta}, where the total energy of a null cone is defined to be the limit of the Hawking mass and the mass loss formula is derived by studying the variation of the Hawking mass. The energy loss formula we derive is the following:
\vskip 10pt
\noindent {\bf Theorem C} (Theorem \ref{thm_loss})
{\it
Let $(M,g,k)$ be an asymptotically hyperbolic initial data set satisfying the vacuum constraint equation. Let $(M,g(t),k(t))$ be the solution to the vacuum Einstein equation with $g(0)=g$ and $k(0)=k$, and with lapse $\sqrt{r^2+1}$ and shift vector $-r e_3$ where $e_3$ is the unit normal vector of the coordinate spheres on $(M, g(t), k(t))$. Let $\Sigma_{r,t}$ be the coordinate spheres on $(M,g(t),k(t))$ and 
\[ M(t) = \lim_{r\to \infty} m_{LY}(\Sigma_{r,t}) \] be the limit.
Along the vacuum Einstein equation, we have
\[ \partial_t  M(t) = \frac{-1}{8 \pi} \int_{S^2} |p_{ab}^{(0)}+ g_{ab}^{(0)}|^2  dS^2.  \]
}
\vskip 10pt
\begin{remark}\label{gauge_condition}
The lapse function and shift vector considered here are natural for asymptotically hyperbolic initial data sets. 
The Minkowski metric can be written as 
\[  g= -dt^2- \frac{2r}{\sqrt{r^2+1}} dr dt + \frac{dr^2}{r^2+1} + r^2 (d\theta^2+\sin^2\theta d\phi^2).\]
Each level set of $t$ is a standard hyperboloid. Let $\Sigma_r$ be a level set of $r$ on a hyperboloid, $e_3$ be its unit normal vector in the hyperboloid, and $e_4$ be the unit normal of the hyperboloid. Thus
\[
\begin{split}
e_3 = &   \sqrt{r^2+1}  \frac{\partial}{\partial r} \\
e_4 = &  \frac{1}{ \sqrt{r^2+1}}\frac{\partial}{\partial t}  + r \frac{\partial}{\partial r}  
\end{split}
\]
and the Killing vector field $\frac{\partial}{\partial t}$ can be written as 
\[ \frac{\partial}{\partial t} = \sqrt{r^2+1} e_4 -r e_3 . \]
Hence, the lapse function is $\sqrt{r^2+1}$ and the shift vector is $-r e_3$. 
\end{remark}
For spacetimes with matter fields, we prove a similar mass loss formula with an additional term from the matter field assuming some natural decay conditions on the  null component of the stress-energy density of the matter field (see Definition \ref{null_condition} and Theorem \ref{thm_loss2}). In particular, for solutions of the Einstein-Maxwell equation, our result resembles the mass loss formula derived in \cite{Zipser} by Zipser. 

While the ADM energy and linear momentum are well-understood for asymptotically flat initial data sets, defining total angular momentum and center of mass is much more complicated and there were several delicate issues concerning the finiteness, well-definedness, and physical validity. In \cite{Chen-Wang-Yau3}, the authors defined new  total conserved quantities for asymptotically flat initial data sets using the limit of quasi-local conserved quantities and proved the finiteness of the new total  angular momentum and center of mass for asymptotically flat initial data sets of order 1. In this article, we compute the limit of quasi-local angular momentum and center of mass at infinity of asymptotically hyperbolic initial data sets. Unlike the asymptotically flat case, we show that, for an asymptotically hyperbolic initial data set, the limit is finite only for the foliation with vanishing linear momentum. Assuming the energy-momentum vector is timelike for a given foliation, we show that there exists another foliation with vanishing linear momentum.  For the new foliation, the total center of mass $C^i$ and total angular momentum $J^i$ (See Definition \ref{total_conserved}) can be explicitly computed in terms of the expansion coefficients of the metric and the second fundamental form.

 \vskip 10pt
\noindent {\bf Theorem D} (Theorem \ref{thm_conserved})
{\it
Given an asymptotically hyperbolic initial data set $(M,g,k)$, for the foliation with vanishing linear momentum, the total center of mass $C^i$ and total angular momentum $J^i$ of $(M,g,k)$ are 
\begin{equation*}
\begin{split}
C^i =& \frac{1}{8 \pi} \int_{S^2} \tilde X^i (2 tr_{S^2} g_{ab}^{(-2)} +g_{rr}^{(-6)}+ \tilde \nabla^a g_{ra}^{(-3)} +  tr_{S^2} p_{ab}^{(-2)}) dS^2 \\
J^i =& \frac{1}{8 \pi} \int_{S^2} \tilde X^i  \tilde\epsilon^{ab}  \tilde \nabla_b (g_{ra}^{(-3)}+p_{ra}^{(-3)}) dS^2 .
\end{split}
\end{equation*}
}

\section{Quasi-local energy-momentum, angular momentum and center of mass}

We recall the definition of quasi-local energy-momentum defined in  \cite{Wang-Yau2} and quasi-local angular momentum and center of mass defined in \cite{Chen-Wang-Yau3}. Let $\Sigma$ be a closed embedded spacelike 2-surface in a spacetime with spacelike mean curvature vector $H$. The data used in the definition of qausi-local mass is the triple $(\sigma,|H|,\alpha_H)$ where $\sigma$ is the induced metric, $|H|$ is the norm of the mean curvature vector and $\alpha_H$ is the connection one form of the normal bundle with respect to the mean curvature vector
\[ \alpha_H(\cdot )=\langle \nabla^N_{(\cdot)}   \frac{J}{|H|}, \frac{H}{|H|}  \rangle  \]
where $J$ is the reflection of $H$ through the incoming  light cone in the normal bundle.

Given an isometric embedding $X:\Sigma\rightarrow \R^{3,1}$ and a constant future timelike unit vector $T_0\in \R^{3,1}$, we consider the projected embedding $\hat{X}$ into the orthogonal complement of $T_0$. We denote the mean curvature of the image by $\widehat{H}$. 

In terms of $\tau= - X \cdot T_0$, the quasi-local energy with respect to the pair $(X, T_0)$ is
\[\begin{split}&E(\Sigma, \tau)=\frac{1}{8 \pi}\int_{\widehat{\Sigma}} \hat{H} d{\widehat{\Sigma}}-\frac{1}{8 \pi }\int_\Sigma \left[\sqrt{1+|\nabla\tau|^2}\cosh\theta|{H}|-\nabla\tau\cdot \nabla \theta -\alpha_H ( \nabla \tau) \right]d\Sigma,\end{split}\] where $\theta=\sinh^{-1}(\frac{-\Delta\tau}{|H|\sqrt{1+|\nabla\tau|^2}})$ and $\nabla$ and $\Delta$ are the gradient and Laplace operator, respectively, with respect to $\sigma$. In particular, $E(\Sigma,0)$ is the same as the Liu-Yau quasi-local mass, $m_{LY}(\Sigma)$.

Assuming the spacetime satisfying the dominant energy condition, the quasi-local energy defined above is non-negative for any admissible pair of  $(X, T_0)$. The quasi-local mass of the surface $\Sigma$ is defined to be the infimum of the quasi-local energy with respect to all admissible pairs. 

Given an isometric embedding $X$ of $\Sigma$ into $\R^{3,1}$, let $H_0$ and $\alpha_{H_0}$  be the mean curvature vector and connection one form of $X(\Sigma)$ in $\R^{3,1}$. We introduce $\rho $ and a $j_a$ as follows: 
   \begin{align} \label{rho}\rho &= \frac{\sqrt{|H_0|^2 +\frac{(\Delta \tau)^2}{1+ |\nabla \tau|^2}} - \sqrt{|H|^2 +\frac{(\Delta \tau)^2}{1+ |\nabla \tau|^2}} }{ \sqrt{1+ |\nabla \tau|^2}}
\\ \label{j_a}
j_a & =\rho {\nabla_a \tau }- \nabla_a [ \sinh^{-1} (\frac{\rho\Delta \tau }{|H_0||H|})]-(\alpha_{H_0})_a + (\alpha_{H})_a. \end{align} 
In terms of these, the quasi-local energy is $\frac{1}{8\pi}\int_\Sigma (\rho+j_a\nabla^a\tau)$.
 
A critical point of the quasi-local energy satisfies the optimal isometric embedding equation. A pair of an embedding $X:\Sigma\hookrightarrow \mathbb{R}^{3,1}$ and an observer $T_0$ satisfies the optimal isometric embedding equation for $(\sigma_{ab}, |H|, (\alpha_H)_a)$ if $X$ is an isometric embedding and 
 \begin{equation}\label{optimal}div_{\sigma} j=0.\end{equation}

We define the quasi-local angular momentum and center of mass for each optimal isometric embedding. Let $(x^0,x^1,x^2,x^3)$ be the standard coordinate of $\R^{3,1}$. We recall that $K$ is a rotational Killing field if $K$ is the image of $x^i \frac{\partial}{\partial x^j}-x^j \frac{\partial}{\partial x^i} $ under a Lorentz transformation. Similarly, $K$ is a boost Killing field  $K$ is the image of $x^0 \frac{\partial}{\partial x^i}+x^i \frac{\partial}{\partial x^0} $ under a Lorentz transformation. 

\begin{definition} 
The quasi-local conserved quantity of $\Sigma$ with respect to an optimal isometric embedding $(X, T_0)$ and a Killing field $K$ is 
\begin{equation}\label{qlcq2}\begin{split}&E(\Sigma, X, T_0, K)=\frac{(-1)}{8\pi} \int_\Sigma
\left[ \langle K, T_0\rangle \rho+j(K^\top) \right]d\Sigma,\end{split}\end{equation} 
where $K^\top$ is the tangential part of $K$ to $X(\Sigma)$. 

Suppose $T_0=A(\frac{\partial}{\partial x^0})$ for a Lorentz transformation $A$, then the quasi-local conserved quantities corresponding to $A(x^i\frac{\partial}{\partial x^j}-x^j\frac{\partial}{\partial x^i}), i<j$ are called the quasi-local angular momentum integral with respect
to $T_0$ and the quasi-local conserved quantities corresponding to $A(x^i\frac{\partial}{\partial x^0}+x^0 \frac{\partial}{\partial x^i}), i=1, 2, 3$ are called the quasi-local center of mass integral with respect to $T_0$. 
\end{definition}


\section{Limit of quasi-local energy of coordinate spheres.} 
We first prove the following lemma about the geometry of coordinate spheres in an asymptotically hyperbolic initial data set. Given an asymptotically hyperbolic initial data set $(M,g,k)$ in a spacetime $N$, such that $g$ and $k$ are the induced metric and second fundamental form of $M$,
respectively.
\begin{lemma}\label{data_expansion}
Let $\Sigma_r$ be coordinate spheres in an asymptotically hyperbolic space. Let $\sigma$, $H$ be the mean curvature vector of $\Sigma_r$ and  $\alpha_{H}$ be the induced metric, mean curvature vector and connection one-form of the normal bundle in mean curvature gauge of $\Sigma_r$ in the spacetime $N$. We have
\begin{equation} \label{eq:expansion_coordinate}
\begin{split}
\sigma_{ab} =& r^2 \tilde \sigma_{ab}+g^{(0)}_{ab} + \frac{ g^{(-1)}_{ab}}{r} +O(r^{-2}), \\ 
|H| = &  \frac{2}{r} - \frac{m   }{r^2} +O(r^{-3}),  \\
(\alpha_{H})_a = & \frac{ \partial_a ( m) }{2r}+O(r^{-2}),
\end{split}
\end{equation}
where $m$ is the mass aspect function defined in Definition \ref{aspect}.

\end{lemma}
\begin{proof} 
The expansion for $\sigma_{ab} $ follows directly from the definition of asymptotically hyperbolic initial data set. The inverse metric is
\[  \sigma^{ab} =\frac{ \tilde \sigma^{ab}}{r^2} - \frac{(g^{(0)})^{ab}}{r^4 }- \frac{(g^{(-1)})^{ab}}{r^5}+ O(r^{-6}).   \]

Let $e_3$ be the unit normal vector of $\Sigma_r$ in $M$ and $e_4$ be the unit normal vector of $M$ in $N$. We compute 
$\langle H, e_3 \rangle$ and $\langle H, e_4 \rangle$. $ e_3$ is given by
\[   e_3 =  \frac{1}{\sqrt{ g_{rr} -|V|_{\sigma}^2 }} (\frac{\partial}{\partial r}  + V^a\frac{\partial}{\partial u^a}  )   \] 
where  $V_a =  -g_{ra}$, and the second fundamental form of $\Sigma_r$ in the direction of $e_3$ is 

\begin{equation}\label{second_ff}
\begin{split}
        \langle \nabla_{\frac{\partial}{\partial u^a}  }   e_3 , \frac{\partial}{\partial u^b}  \rangle = & \sqrt{r^2+1 - \frac{ g^{(-5)}_{rr}}{r}}  \langle \nabla_{\frac{\partial}{\partial u^a}  }  \frac{\partial}{\partial r}  , \frac{\partial}{\partial u^b}  \rangle  + O(r^{-2}) \\
=&  \sqrt{r^2+1 - \frac{ g^{(-5)}_{rr}}{r}} (r \tilde \sigma_{ab} - \frac{g^{(-1)}_{ab}}{2 r^2})   + O(r^{-2})   .
\end{split}
\end{equation}
Hence, 
\[\begin{split}
       - \langle H, e_3 \rangle  
= &  \sigma^{ab}   \sqrt{r^2+1 - \frac{ g^{(-5)}_{rr}}{r}} (r \tilde \sigma_{ab} - \frac{g^{(-1)}_{ab}}{2 r^2})  + O(r^{-4})  \\
= &  r(1+ \frac{1}{2r^2} - \frac{ g^{(-5)}_{rr}}{2r^3})(r \tilde \sigma_{ab} - \frac{g^{(-1)}_{ab}}{2 r^2}) (\frac{ \tilde \sigma^{ab}}{r^2} - \frac{(g^{(0)})^{ab}}{r^4 }- \frac{(g^{(-1)})^{ab}}{r^5}) + O(r^{-4}) \\
= & 2+ \frac{1}{r^2} - \frac{ g^{(-5)}_{rr} + \frac{3}{2} tr_{S^2} g^{(-1)}_{ab} }{r^3} + O(r^{-4}).
\end{split}
\]

For $\langle H, e_4 \rangle$, we have
\[  
\begin{split}
   -\langle H, e_4 \rangle
= &   \sigma^{ab}( \sigma_{ab}  + p_{ab} ) \\
=& 2 + \frac{ tr_{S^2} p^{(-1)}_{ab} }{r^3} + O(r^{-4}).
\end{split}
\]
As a result, 
\[  
\begin{split}
   |H|^2  =&  \langle H, e_3 \rangle ^2 - \langle H, e_4 \rangle ^2 \\ 
= & \frac{4}{r^2} - \frac{4  g^{(-5)}_{rr} +4tr_{S^2} p^{(-1)}_{ab} + 6 tr_{S^2}  g^{(-1)}_{ab} }{r^3}+O(r^{-4})
\end{split}
\]
and the expansion for $|H|$ follows. 

Recall from \cite{Wang-Yau3}, we have
\[  ( \alpha_{H})_a =- k(e_3, \partial_a) + \nabla_a \theta \]
where 
\[  \sinh (\theta) = \frac{- \langle H, e_4 \rangle }{|H|}. \]
By the definition of asymptotically hyperbolic initial data set, we have $  k(e_3, \partial_a) =O(r^{-2}).$
Using the expansion for $\langle H, e_4 \rangle $ and $|H|$, we conclude that 
\[  e^\theta =  2r + ( tr_{S^2} p^{(-1)}_{ab}+  \frac{3}{2} tr_{S^2}   g^{(-1)}_{ab}  + g^{(-5)}_{rr})  +  O(r^{-1}) \]
which implies
\[  \nabla_a \theta = \frac{ \partial_a ( tr_{S^2} p^{(-1)}_{ab}+  \frac{3}{2} tr_{S^2}   g^{(-1)}_{ab}  + g^{(-5)}_{rr})  }{2r} +O(r^{-2}).  \]
\end{proof}

In \cite{Wang-Yau3}, it is shown that for an asymptotically flat initial data set, the limit of the quasi-local energy of coordinate spheres with reference isometric embedding into $\R^3$ recovers the ADM energy-momentum 4-vector. We prove a similar result for  asymptotically hyperbolic initial data sets. Let $X_r$ be the isometric embedding of $\Sigma_r$ into the totally geodesics $\R^3$ in $\R^{3,1}$.  
\begin{theorem}  \label{limit_vector}
Given an asymptotically hyperbolic initial data set $(M,g,k)$, the limit of the quasi-local energy of $\Sigma_r$ with reference embedding $X_r$ into $\R^3$ is a linear function dual to the energy-momentum of $(M,g,k)$.
\end{theorem}
\begin{proof}
Let $H_0$ be the mean curvature of the image of $X_r$ in $\R^3$. From the expansion of the induced metric and the result in \cite{Chen-Wang-Yau1}, it follows that $H_0 = \frac{2}{r} + O(r^{-3})$ and thus 
\[  \lim_{r \to \infty }  \frac{|H|}{H_0} = 1.
 \]
By Corollary 2.1 of \cite{Wang-Yau3}, the limit of 
\[\lim_{r \to \infty }   (\Sigma_r , X_r, T_0) \] 
is a linear function dual to the four vector $(e, p^i)$ where
\[
\begin{split}
e =  & \frac{1}{8 \pi} \lim_{r \to \infty }   \int_{\Sigma_r} (H_0 -|H| ) d \Sigma_r  \\
p^i = &  \frac{1}{8 \pi} \lim_{r \to \infty }   \int_{\Sigma_r}  X^i \nabla^a (\alpha_H)_a d \Sigma_r.
\end{split}
\]
By equation \eqref{eq:expansion_coordinate},  
\[   \frac{1}{8 \pi}   \int_{\Sigma_r} (H_0 -|H| ) d \Sigma_r  = \frac{1}{8 \pi} \int_{S^2}( \frac{3}{2}tr_{S^2} \tilde g^{(-1)}_{ab}+  tr_{S^2}p^{(-1)}_{ab} +g^{(-5)}_{rr})  dS^2  + O(r^{-1})\]
and 
\[ \frac{1}{8 \pi}  \int_{\Sigma_r}  X^i \nabla^a (\alpha_H)_a d \Sigma_r = -\frac{1}{8 \pi} \int_{S^2} \tilde X^i( \frac{3}{2}tr_{S^2}g^{(-1)}_{ab}+  tr_{S^2}p^{(-1)}_{ab} +g^{(-5)}_{rr})  dS^2 + O(r^{-1}) .  \]
This finishes the proof of the theorem.
\end{proof}
\begin{remark}
The data $(\sigma,|H|,\alpha_H)$ has the same form of expansion as that of \cite{Chen-Wang-Yau1} and section 3 and 4 of \cite{Chen-Wang-Yau1} can be applied to the asymptotically hyperbolic case as well. Assuming the total energy momentum vector $(E_{AH}, P^i_{AH})$ is timelike, Let 
\[ M_{AH} = \sqrt{E_{AH}^2 - \sum_i (P^i_{AH})^2}.  \]
There is a local minimum $(X(r),T_0(r))$ of the quasi-local energy with 
\[
\begin{split}
X^0(r) = & (X^0)^{(0)} + O(r^{-1})\\
X^i (r)= & r\tilde X^i+\frac{(X^i)^{(-1)}}{r}+ O(r^{-2})\\
T_0(r) = & (a^0,a^1,a^2,a^3) + \frac{T_0^{(-1)}}{r}+  O(r^{-2}).
\end{split} 
\]
where $ (a^0,a^1,a^2,a^3)$ is the unit timelike vector aligned with the total energy-momentum 4-vector. 
 \[M_{AH}( a^0,a^1,a^2,a^3) = (E_{AH}, P^i_{AH}).\] 
Moreover,
\[ \lim_{r \to \infty} E(\Sigma_r, X(r),T_0(r)) =  M_{AH}. \]
\end{remark}

\section{Spacetime positive mass theorem}
Theorem \ref{limit_vector} and the positivity of Liu-Yau quasi-local mass implies $E_{AH} \ge 0$ if $(M,g,k)$ satisfies the dominant energy condition. We  prove that the energy-momentum vector is non-spacelike by evaluating the limit of quasi-local energy on boosted coordinate spheres in $M$.

The hyperbolic space can be isometrically embedded into $\R^{3,1}$ with metric $-dt^2+\sum_{i=1}^3 (dx^i)^2$ as the set
\[ \h^3 = \{ (t,x^i) \, | \, t^2=1+\sum_i  (x^i)^2, t>0 \}.\]
The Lorentz group of $\R^{3,1}$ acts on $\h^3$ by isometry. In particular, we consider the isometries corresponding to boost elements of the Lorentz group. Let $(M, g, k)$ be an asymptotically hyperbolic initial data set and $(r, u^1, u^2)$ be the asymptotically hyperbolic coordinate system at
infinity. For each future directed unit timelike vector $(a^0,a^i)$ in $\R^{3,1}$, we consider the family of surfaces $\Sigma_R$, $R>>1$  on $M$ defined by
 \[ \Sigma_R= \{ (r, u^1, u^2):  r= R\tilde F(u^1, u^2)  \}  \]
where 
\[  \tilde F = \frac{1}{a^0+ \sum_i a^i \tilde X^i} \] and $\tilde{X}^i(u^1, u^2), i=1, 2, 3$ are the three standard  coordinate functions on $S^2$. They are all $-2$ eigenfunctions with respect to the standard
round metric $\tilde{\sigma}_{ab}$ on $S^2$. Notice that $\tilde F^2 \tilde
 \sigma_{ab}$ is isometric to $\tilde \sigma_{ab}$ and $\tilde F$ satisfies the following constant Gauss curvature equation
\[ \frac{1}{\tilde F^2} (1 - \widetilde \Delta \ln \tilde F) =1  \]
where $\widetilde \Delta $ is the Laplace operator with respect to the metric $\tilde \sigma_{ab}$.
\begin{lemma}
Let $\Sigma_R$ be the above family of surfaces. Let $\sigma$ and  $H$ be the induced metric and the mean curvature vector  of $\Sigma_R$, respectively. Then,
\begin{equation} \label{eq:expansion_change}
\begin{split}
\sigma_{ab} =& (\tilde F ^2 R^2) \tilde \sigma_{ab} +g^{(0)}_{ab}+ \frac{\tilde F_a \tilde F_b}{\tilde F^2}+ \frac{g_{ab}^{(-1)}}{\tilde F R} +O(R^{-2}) \\ 
|H| = &  \frac{2}{R} - \frac{g_{rr}^{(-5)} +tr_{S^2} p_{ab}^{(-1)} + \frac{3}{2} tr_{S^2} g_{ab}^{(-1)} }{\tilde F^3 R^2} +O(R^{-3}).
\end{split}
\end{equation}
\end{lemma}
\begin{proof} 
The tangent space of $\Sigma_R$ is spanned by $ \frac{\partial }{\partial u^a} + r_a \frac{\partial }{\partial r}, a=1, 2$.
It follows immediately from the definition of asymptotically hyperbolic initial data sets that 
\[ 
\sigma_{ab} = (\tilde F ^2 R^2) \tilde \sigma_{ab}+g^{(0)}_{ab} + \frac{\tilde F_a \tilde F_b}{\tilde F^2}+ \frac{g_{ab}^{(-1)}}{\tilde F R} +O(R^{-2}).
 \]
As a result, we also have
\[ 
\sigma^{ab} = \tilde F^{-2}R^{-2}( \tilde \sigma^{ab} -\frac{(g^{(0)})^{ab}}{ \tilde F^2 R^2}-\frac{\tilde F^a \tilde F^b}{ \tilde F^4 R^2}- \frac{(g^{(-1)})^{ab}}{\tilde F^3 R^3} ) +O(R^{-6}).
 \]
Let $e_3$ be the unit normal of $\Sigma_R$ in $M$ and $e_4$ be the unit normal of $M$ in $N$.  
We compute \[  
\begin{split}
 -  \langle H, e_4 \rangle 
= &   \sigma^{ab}( \sigma_{ab}  + p_{ab} + r_a r_b p_{rr} + r_a p_{rb} + r_b p_{ra} ) \\
=& 2 + \frac{ tr_{S^2} p^{(-1)}_{ab} }{\tilde F^3 R^3} + O(R^{-4}).
\end{split}
\]
To compute $\langle H, e_3\rangle$, start with 
\[  e_3= \frac{1}{  \sqrt{ g_{rr}+ V^aV^bg_{ab} -2V^ag_{ar} } }(\frac{\partial }{\partial r} + V^a  \frac{\partial }{\partial u^a} ) \]
where 
\[  V_a = \frac{ - r_a}{1+r^2} +O(R^{-3}). \]

By definition,
\[  
\begin{split}
   & - \langle H, e_3 \rangle  \\
= &  \frac{\sigma^{ab}}{  \sqrt{ g_{rr}+ V^dV^eg_{de} -2V^dg_{dr} } }  \langle \nabla_{\frac{\partial}{\partial u^a} + r_a \frac{\partial }{\partial r} }  (\frac{\partial }{\partial r}-\frac{r^c}{1+r^2} \frac{\partial }{\partial u^c} )  , \frac{\partial}{\partial u^b}  + r_b \frac{\partial }{\partial r} \rangle + O(R^{-4}).
\end{split}
\]

Plugging in $r=R \tilde{F}(u^1,u^2)$, we deduce
\[   \frac{1}{  \sqrt{ g_{rr}+ V^aV^bg_{ab} -2V^ag_{ar} } } = R\tilde F \left [ 1+ \frac{\tilde F^2 - |\tilde \nabla \tilde F|^2}{2R^2 \tilde F^4} - \frac{g_{rr}^{(-5)}}{2R^3 \tilde F^3} \right ] + O(R^{-3}). \]

Continuing the calculation, 
\[  
\begin{split}
   &   \langle \nabla_{\frac{\partial}{\partial u^a} + r_a \frac{\partial }{\partial r} } ( \frac{\partial }{\partial r} - \frac{r^c}{1+r^2} \frac{\partial }{\partial u^c})   , \frac{\partial}{\partial u^b} + r_b \frac{\partial }{\partial r} \rangle  \\
= & \frac{1}{2} \partial_r g_{ab} - \langle \nabla_{\frac{\partial}{\partial u^a} } \frac{r^c}{1+r^2} \frac{\partial }{\partial u^c}   , \frac{\partial}{\partial u^b}  \rangle + r_a r_b \frac{1}{2} \partial_r (\frac{1}{1+r^2})\\
 & - \langle \nabla_{r_a \frac{\partial}{\partial r} } \frac{r^c}{1+r^2} \frac{\partial }{\partial u^c}   ,   \frac{\partial}{\partial u^b}  \rangle 
    - \langle \nabla_{\frac{\partial}{\partial u^a} } \frac{r^c}{1+r^2} \frac{\partial }{\partial u^c}  ,+ r_b \frac{\partial }{\partial r} \rangle + O(R^{-3}) \\
=  &  r \sigma_{ab} - \frac{g_{ab}^{(-1)}}{2r^2} + \frac{r_a r_b}{r^3} + \nabla_a \nabla_b \frac{1}{r} + O(R^{-3}) \\
=& R \tilde F \tilde \sigma_{ab} + \frac{1}{R} (\frac{\tilde F_a \tilde F_b}{\tilde F^3} + \nabla_a\nabla_b\frac{1}{\tilde F}) - \frac{g_{ab}^{(-1)}}{2R^2 \tilde F^2}+ O(R^{-3}) .
\end{split}
\]
As a result, 
\[  
\begin{split}
   &  -\langle H, e_3 \rangle  \\
= & 2 +  \frac{1}{R^2}  \left (\frac{1}{\tilde F^2} - \frac{|\tilde \nabla \tilde F|^2}{\tilde F^4} +\frac{1}{\tilde F}\widetilde \Delta \frac{1}{\tilde F}\right ) - \frac{1}{R^3} \left (\frac{g_{rr}^{(-5)}}{\tilde F^3} +\frac{3 tr_{S^2} g_{ab}^{(-1)} }{2 \tilde F^3} \right ) + O(R^{-4}) \\
= & 2 +  \frac{1}{R^2} - \frac{1}{R^3} \left (\frac{g_{rr}^{(-5)}}{\tilde F^3} +\frac{3 tr_{S^2} g_{ab}^{(-1)} }{2 \tilde F^3} \right ) + O(R^{-4}).
\end{split}
\]
In the last equality, we used that 
\[  \frac{1}{\tilde F^2} - \frac{|\tilde \nabla \tilde F|^2}{\tilde F^4}   +\frac{1}{\tilde F}\widetilde \Delta \frac{1}{\tilde F}=\frac{1}{\tilde F^2} (1 - \widetilde \Delta \ln \tilde F) =1.   \]
Therefore,
\[  
\begin{split}
   |H|^2  =&  \langle H, e_3 \rangle ^2 - \langle H, e_4 \rangle ^2 \\ 
= & \frac{4}{R^2} - \frac{4 g_{rr}^{(-5)} +4tr_{S^2} p^{(-1)}_{ab} +6 tr_{S^2} g_{ab}^{(-1)} }{R^3 \tilde F^3}+O(R^{-4}).
\end{split}
\]
\end{proof}
We prove the positive mass theorem for asymptotically hyperbolic initial data set in the following.
\begin{theorem} \label{thm_positivity}
Suppose  $(M,g,k)$ is an asymptotically hyperbolic initial data set satisfying the dominant energy condition, then its  energy-momentum 4-vector is future directed non-spacelike. Namely,
\[   E_{AH}  \ge \sqrt{\sum_i (P^i_{AH})^2 } . \]
\end{theorem} 
\begin{proof} 
For any future directed unit timelike vector $(a^0,a^i)$ in $\R^{3,1}$,  let $\tilde F= \frac{1}{ a^0 + \sum_i a^i \tilde X^i}$ and 
 $\Sigma_R$  be the surfaces in  $M$ defined by
 \[ \Sigma_R= \{ (r, u^1, u^2): r=R \tilde F (u^1, u^2)  \}  \]
By \eqref{eq:expansion_change}, the induced metric $\sigma_{ab}$ and mean curvature vector $H$ of $\Sigma_R$ satisfy
\begin{equation} 
\begin{split}
\sigma_{ab} =& (\tilde F ^2 R^2) \tilde \sigma_{ab} + O(1) \\ 
|H| = &  \frac{2}{R} - \frac{g_{rr}^{(-5)}+ tr_{S^2} p^{(-1)}_{ab} + \frac{3}{2} tr_{S^2} g_{ab}^{(-1)} }{\tilde F^3 R^2} +O(R^{-3}).
\end{split}
\end{equation}
Let  $H_0$ be the mean curvature of the image of isometric embedding of $\Sigma_R$ in $\R^3$. Recall that the Liu-Yau mass of the surface $\Sigma_R$ is
\[ m_{LY}(\Sigma_R) = \frac{1}{8 \pi}\int_{\Sigma_R} (H_0 - |H| )d \Sigma_R. \]
We have
\[  H_0= \frac{2}{R} + O(R^{-3}). \]
since $(\tilde F ^2 R^2) \tilde \sigma_{ab}$ is isometric to $ R^2 \tilde \sigma_{ab}$. As a result,
\begin{equation}\label{equvariant}
 \int_{\Sigma_R} (H_0 - |H|) d \Sigma_R  = \int_{S^2 }  \frac{ g_{rr}^{(-5)}+ tr_{S^2} p_{ab}^{(-1)} + \frac{3}{2} tr_{S^2} g_{ab}^{(-1)}  }{\tilde F } dS^2 + O(R^{-1}). \end{equation}
By  the positivity of the Liu-Yau mass \cite{Liu-Yau}, we have
\[   \int_{S^2 }  \frac{ g_{rr}^{(-5)}+ tr_{S^2} p_{ab}^{(-1)} + \frac{3}{2} tr_{S^2} g_{ab}^{(-1)}  }{\tilde F }  dS^2 \ge 0. \]
Moreover, by Theorem \ref{limit_vector},
\[  \frac{1}{8 \pi}  \int_{S^2 } \frac{ g_{rr}^{(-5)}+ tr_{S^2} p_{ab}^{(-1)} + \frac{3}{2} tr_{S^2} g_{ab}^{(-1)}  }{\tilde F } dS^2 = a^0 E_{AH} - \sum_i a^i P^i_{AH}. \]
Hence,
\[  a^0 E_{AH} - \sum_i a^i P^i_{AH}  \ge 0  \]
for  any future directed unit timelike vector $(a^0,a^i)$ in $\R^{3,1}$. It follows that $(E_{AH}, P^i_{AH})$ is future-directed non-spacelike.
\end{proof}
\begin{corollary}\label{cor_vanish_momentum}
Assuming the energy-momentum vector $(E_{AH}, P^i_{AH})$ for a given foliation is timelike, then there exists a foliation with vanishing linear momentum.
\end{corollary}
\begin{proof}
From equation \eqref{equvariant}, it follows that the energy component of the energy momentum vector transformed equivariantly under boost. As a result, if  the energy-momentum vector $(E_{AH}, P^i_{AH})$ for a given foliation is timelike, then we can align the unit vector $(a^0, a^i)$ with the energy momentum and the new foliation corresponding to  $(a^0, a^i)$ have vanishing linear momentum.
\end{proof}
\section{Energy loss formula  for asymptotically hyperbolic manifolds}
Energy loss formula at null infinity for solutions of the vacuum Einstein equation was first studied in \cite{Bondi-Burg-Metzner} and \cite{Trautman}. In  \cite{sta}, a similar energy loss formula at null infinity is proved as a consequence of the global nonlinear stability of the Minkowski space, for any global development of asymptotically flat vacuum initial data sufficiently close to that of the Minkowski space. The source of the energy loss is the traceless part of the inward null second fundamental form of the coordinate spheres. In this section, we study the mass loss formula for asymptotically hyperbolic initial data sets under the Einstein equation. By Theorem \ref{limit_vector}, we can evaluate the total energy at infinity of an asymptotically hyperbolic initial data set with respect to a given foliation of coordinate spheres using the limit of the Liu-Yau mass. We first prove the following energy loss formula for a vacuum spacetime.
\begin{theorem}\label{thm_loss}
Let $(M,g,k)$ be an asymptotically hyperbolic initial data set satisfying the vacuum constraint equation. Let $(M,g(t),k(t))$ be the solution to the vacuum Einstein equation with $g(0)=g$ and $k(0)=k$, and with lapse $\sqrt{r^2+1}$ and shift vector $-r e_3$ where $e_3$ is the unit normal vector of the coordinate spheres. Let $\Sigma_{r,t}$ be the coordinate spheres on $(M,g(t),k(t))$. Let 
\[ M(t) = \lim_{r\to \infty} m_{LY}(\Sigma_{r,t}), \]
then
\[ \partial_t M(t)|_{t=0} = \frac{-1}{8 \pi} \int_{S^2} |p^{(0)}_{ab}+ g^{(0)}_{ab}|^2  dS^2.\]
\end{theorem}  
\begin{proof}
Consider the timelike hypersurface
\[  \Omega_s  =  \cup_{0 \le t \le s } \Sigma_{r,t}.\]
Let $V$ be the unit normal vector of $\Sigma_{r,t}$ in  $\Omega_s$. 
We have $V =\sqrt{r^2+1} e_4  -r e_3$ on $\Sigma_{r,t}$ and
\begin{equation} \label{small_deform}
\langle H , V \rangle = O(r^{-2})
 \end{equation} by the calculation in Lemma \ref{data_expansion}.
 Let $\gamma_{ij}$ and  $\Theta_{ij}$ be the metric and second fundamental form of $\Omega_s$ and  
$\pi_{ij}= \Theta_{ij} -( tr _{\gamma} \Theta) \gamma_{ij}$ be the conjugate momentum. Let $\nu$ be the outward unit normal of $\Omega_s$. Let $H(\Sigma_{r,0})$ and $H(\Sigma_{r,s})$ be the mean curvature vector of $\Sigma_{r,0}$ and $\Sigma_{r,s}$, respectively, in the spacetime and let $H_0(\Sigma_{r,0})$ and $H_0(\Sigma_{r,s})$  be the mean curvature of the image of isometric embedding of $\Sigma_{r,0}$ and $\Sigma_{r,s}$, respectively, into $\R^3$.
Integrating by parts,
\[  \int_{\Sigma_{r,s}} \pi(V,V)   d \Sigma_{r,s} =   \int_{\Sigma_{r,0}} \pi(V,V)   d \Sigma_{r,s} + \int_{\Omega_s} \pi_{ij}V_{i,j}. \]
By the definition of the conjugate momentum, 
\[
\begin{split}
 \pi(V,V)|_{\Sigma_{r,0}}  = &\langle H (\Sigma_{r,0}), \nu \rangle \\
 \pi(V,V)|_{\Sigma_{r,s}}  =  & \langle H (\Sigma_{r,s}), \nu \rangle.
\end{split}
\]
From \eqref{small_deform}, we have
\[
\begin{split}
|H(\Sigma_{r,0})| = &\langle H (\Sigma_{r,0}), \nu \rangle  + O(r^{-3})\\
|H(\Sigma_{r,s})| =  & \langle H (\Sigma_{r,s}), \nu \rangle + O(r^{-3})
\end{split}
\]
and thus
\[   \int_{\Sigma_{r,s}} |H(\Sigma_{r,s})| d \Sigma_{r,s} =   \int_{\Sigma_{r,0}}|H(\Sigma_{r,0})|d \Sigma_{r,0} + \int_{\Omega_s} \pi_{ij}V_{i,j}   + O(r^{-1}). \]
From  \eqref{small_deform}, the area  $A(r,s)$ of $\Sigma_{r,s}$ satisfies
\[  A(r,s) = A(r,0) + O(1).\]
From Lemma 2.1 of \cite{Fan-Shi-Tam}, we have
\[   \int_{\Sigma_{r,s}}H_0(\Sigma_{r,s}) d \Sigma_{r,s} =  4 \pi r + \frac{A(r,s)}{r} + O(r^{-1}). \]
We conclude that 
\[   \int_{\Sigma_{r,s}}H_0(\Sigma_{r,s}) d \Sigma_{r,s} =   \int_{\Sigma_{r,0}} H_0(\Sigma_{r,0})  d \Sigma_{r,0}  + O(r^{-1}). \]

It follows that 
\[   m_{LY} (\Sigma_{r,s}) = m_{LY} (\Sigma_{r,0}) +  \frac{1}{8 \pi} \int_{\Omega_s} \pi_{ij}V_{i,j}   + O(r^{-1})   \]
and 
\[ 
\begin{split}
\partial_t  m_{LY} (\Sigma_{r,t}) |_{t=0} 
=  & \frac{1}{8 \pi} \int_{\Sigma_{r,0}} \pi_{ij}V_{i,j}  d \Sigma_{r,0} + O(r^{-1})\\
=  & \frac{1}{8 \pi} \int_{\Sigma_{r,0}} \pi^{ab}V_{a,b}  d \Sigma_{r,0} + O(r^{-1}).\\
\end{split}
\]
Let $h^{(3)}_{ab}$ and  $h^{(4)}_{ab}$ be the second fundamental form of $\Sigma_r$ in the direction of $e_3$ and $e_4$. Furthermore, 
let $\hat  h^{(3)}_{ab}$ and $\hat  h^{(4)}_{ab}$ be their traceless part. 
\[   
\pi^{ab}V_{a,b}  = \pi^{ab}( \sqrt{r^2+1}h^{(4)}_{ab} -  r h^{(3)}_{ab} )  
 \]
since  $\pi$ is symmetric and the symmetrization of $V_{a,b}$ is the second fundamental form of $\Sigma_r$ in the direction of $\sqrt{r^2+1} e_4  -r e_3$.
Equation \eqref{small_deform} implies that 
\[  \sigma^{ab} ( \sqrt{r^2+1}h^{(4)}_{ab} -  r h^{(3)}_{ab} ) = O(r^{-2}).  \]

From the definition of conjugate momentum,
\[   
\begin{split}
     & \pi^{ab}( \sqrt{r^2+1}h^{(4)}_{ab} -  r h^{(3)}_{ab} )  \\
=  & [\sqrt{r^2+1}(h^{(3)})^{ab} -  r (h^{(4)})^{ab} - (tr_{\gamma} \Theta )g^{ab}   ]( \sqrt{r^2+1}h^{(4)}_{ab} -  r h^{(3)}_{ab} ) + O(r^{-3}) \\
=  &  (\sqrt{r^2+1}(\hat h^{(3)})^{ab} -  r (\hat h^{(4)})^{ab}   )( \sqrt{r^2+1}h^{(4)}_{ab} -  r h^{(3)}_{ab} ) + O(r^{-3}) \\
=  &  (\sqrt{r^2+1}(\hat h^{(3)})^{ab} -  r (\hat h^{(4)})^{ab}   )( \sqrt{r^2+1}\hat h^{(4)}_{ab} -  r \hat h^{(3)}_{ab} ) + O(r^{-3}).
\end{split}
\]
From equation \eqref{second_ff}, we have
\[  h^{(3)}_{ab} =  \sqrt{r^2+1 - \frac{g_{rr}^{(-5)}}{r}} r \tilde \sigma_{ab} + O(r^{-1})   . \]
By Definition  \ref{a_h_coordinates}, we have 
\[
\begin{split}
  \sigma_{ab} =&  r^2 \tilde \sigma_{ab}+ g_{ab}^{(0)}+ O(r^{-1})\\
 h^{(4)}_{ab} = &   \sigma_{ab}  + p_{ab}^{(0)} + O(r^{-1})    
\end{split}
\]
where $tr_{S^2}p_{ab}^{(0)}=tr_{S^2}g_{ab}^{(0)} = 0 $.  
 It follows that 
\[
\begin{split}
 \hat h^{(3)}_{ab}  =&  - g_{ab}^{(0)} + O(r^{-1})\\
 \hat h^{(4)}_{ab}  = &  p_{ab}^{(0)}+ O(r^{-1}).
\end{split}
\]
As a result, we have
\[    (\sqrt{r^2+1}(\hat h^{(3)})^{ab} -  r (\hat h^{(4)})^{ab}   )( \sqrt{r^2+1}\hat h^{(4)}_{ab} -  r \hat h^{(3)}_{ab} )= \frac{-| p_{ab}^{(0)}+ g_{ab}^{(0)} |^2}{r^2} +  O(r^{-3}) \]
and 
\[ 
   \partial_t  m_{LY} (\Sigma_{r,t}) |_{t=0} 
= \frac{-1}{8 \pi} \int_{S^2} \frac{|p_{ab}^{(0)}+ g_{ab}^{(0)}|^2}{r^2}  d S^2 + O(r^{-1}).  \]
The theorem follows from taking the limit as $r$ approaches infinity.
\end{proof}
In the remaining part of this section, we generalize the above mass loss formula to spacetime with matter fields. Instead of considering specific matter fields, we require the matter fields to satisfies only the dominant energy condition and the following decay condition for the stress-energy density $T_{\mu \nu}$
\begin{definition} \label{null_condition}
Let $(M,g,k)$ be an asymptotically hyperbolic initial data set in a spacetime $N$ with some matter field. Let $e_3$ and $e_4$ be the unit normal of $\Sigma_r$ in  $M$  and the unit normal of $M$ in $N$. We say that the  stress-energy density $T_{\mu \nu}$ satisfies the null decay condition if
\[
\begin{split}
T(e_3+e_4,e_3+e_4) =  &  O(r^{-3}), \\
T(e_3-e_4,e_3-e_4) =  &  O(r^{-4}). \\
\end{split}
\]
\end{definition}
The null decay condition is natural for the global development of an asymptotically flat initial data set. For example, see \cite{Zipser} for the asymptotics at null infinity for solutions of Einstein-Maxwell equation. In particular, the null decomposition of the the stress-energy density satisfies
\[
\begin{split}
T(L,L) =  &  O(r^{-5}), \\
T(\underline{ L}, \underline{ L}) =  &  O(r^{-2}), \\
\end{split}
\]
and the leading term of $T(\underline{ L}, \underline{ L})$ has an additional contribution to the mass loss formula \cite{Zipser}. By considering the standard hyperbolic space embedded in the Minkowski space (see Remark \ref{gauge_condition}), it is natural to identify $L$ with $\frac{e_3+e_4}{r}$ and $\underline{ L}$ with $r(e_3+e_4)$ and assume the null condition in Definition \ref{null_condition}.

The dominant energy condition implies that both $T(e_3+e_4,e_3+e_4)$ and $T(e_3-e_4,e_3-e_4)$ are non-negative. Let 
\begin{equation}\label{T34} T(e_3-e_4,e_3-e_4) = \frac{F^2}{r^4} + O(r^{-5}). \end{equation}
We have
\begin{theorem}\label{thm_loss2}
Let $(M,g,k)$ be an asymptotically hyperbolic initial data set  satisfying the dominant energy condition equation. Let $(M,g(t),k(t))$ be the solution to the Einstein equation with $g(0)=g$ and $k(0)=k$ with lapse $\sqrt{r^2+1}$ and shift vector $-r e_3$ where $e_3$ is the unit normal vector of the coordinate spheres. 
Let $\Sigma_{r,t}$ be the coordinate spheres on $(M,g(t),k(t))$. Let 
\[ M(t) = \lim_{r\to \infty} m_{LY}(\Sigma_{r,t}). \]
Suppose the stress energy density satisfies both the null condition and the dominant energy condition. Then
\[ \partial_t  M(t)|_{t=0} = \frac{-1}{8 \pi} \int_{S^2} (|p^{(0)}_{ab}+g^{(0)}_{ab}|^2 +F^2 )dS^2 \] where $F$ is given by \eqref{T34}.
\end{theorem}  
\begin{proof}
The proof is similar to that of the  previous theorem. Due to the non-zero stress energy density, we have
\[ 
\begin{split}
    \int_{\Sigma_{r,s}} \pi(V,V)   d \Sigma_{r,s} 
=  & \int_{\Sigma_{r,0}} \pi(V,V)   d \Sigma_{r,s} + \int_{\Omega_s} \pi_{ij}V_{i,j} + \int_{\Omega_s} \nabla^i \pi_{ij}V_{j}   \\
=  & \int_{\Sigma_{r,0}} \pi(V,V)   d \Sigma_{r,s} + \int_{\Omega_s} \pi_{ij}V_{i,j} + \int_{\Omega_s} T(V,\nu)   \\
\end{split} \]
 where $\nu$ is the unit normal of the timelike hypersurface $\Omega_s$. As a result,
\[ 
      \partial_t  m_{LY} (\Sigma_{r,t}) |_{t=0} =
 \frac{1}{8 \pi} \int_{\Sigma_{r,0}} (\pi^{ab}V_{a,b}+  T(V,\nu)  )  d \Sigma_{r,0}     + O(r^{-1})
\]
The first term can be treated as before. For the second term, 
\[ 
\begin{split}
    & T(V,\nu)  \\
=  & T\left(\sqrt{1+r^2}e_4 - r e_3 ,\sqrt{1+r^2}e_3 - r e_4 \right) \\
=&\frac{1}{4} [ - (\sqrt{1+r^2} +r)^2 T(e_4-e_3 , e_4- e_3)  + (\sqrt{1+r^2} - r)^2 T(e_4+e_3, e_4+e_3)  ]\\
=  &  \frac{- F^2}{r^2} + O(r^{-3})
\end{split}
\]
\end{proof}
\section{Finiteness of total angular momentum and center of mass}
Let $(\sigma,|H|,\alpha_H)$ be the induced metric, norm of mean curvature vector and connection one-form in mean curvature gauge on the coordinate spheres $\Sigma_r$. The data admits the same expansion as data of the coordinate spheres at the infinity of an asymptotically flat initial data set of order 1 studied in \cite{Chen-Wang-Yau3}. Thus, we solve the optimal embedding equation and define the total angular momentum and center of mass of $(M,g,k)$ as the limit of quasi-local angular momentum and center of mass on the coordinate spheres in the same manner as in Section 6  of \cite{Chen-Wang-Yau3}.

Denote the expansion of the data $(\sigma,|H|, \alpha_H)$ on the coordinate spheres by  
\begin{equation}
\begin{split}
\sigma = & r^2\tilde \sigma_{ab} + O(r^{-1})\\
|H|= & \frac{2}{r}+ \frac{h^{(-2)}}{r^2}+ \frac{h^{(-3)}}{r^3}+ O(r^{-4})\\
(\alpha_{H})_a= & \frac{(\alpha_H^{(-1)})_a}{r}+ \frac{(\alpha_H^{(-2)})_a}{r^2} + O(r^{-3}).
\end{split}
\end{equation}

From \cite{Chen-Wang-Yau2} and \cite{Chen-Wang-Yau3} for the solution of the optimal embedding equation, we have the following expansion for the solution $(X(r),T_0(r))$ of the optimal embedding equation
\begin{equation} 
\begin{split}
X(r)^0= & (X^0)^{(0)} + O(r^{-1})\\
X(r)^i = & r\tilde X^i+(X^i)^{(0)}+\frac{(X^i)^{(-1)}}{r}+ O(r^{-2})\\
T_0(r) = & (a^0,a^1,a^2,a^3) + \frac{T_0^{(-1)}}{r}+  O(r^{-2}).
\end{split}
\end{equation}
We have corresponding expansions for the data of the image of the isometric embedding
\begin{equation}
\begin{split}
|H_0|= & \frac{2}{r}+ \frac{h_0^{(-2)}}{r^2}+ \frac{h_0^{(-3)}}{r^3}+ O(r^{-4})\\
(\alpha_{H_0})_a= & \frac{(\alpha_{H_0}^{(-1)})_a}{r}+ \frac{(\alpha_{H_0}^{(-2)})_a}{r^2} + O(r^{-3}),
\end{split}
\end{equation}
and the expansion for $\rho$ as defined in equation \eqref{rho}
\[  \rho = \frac{\rho^{(-2)}}{r^2}+ \frac{\rho^{(-3)}}{r^3}+ O(r^{-4}).  \]
The  metric expansion implies that $h_0^{(-2)}=0$,  $(X^i)^{(0)} =0$, and
$ \rho^{(-2)} = \frac{- h^{(-2)} }{a^0}. $

The leading order term of the optimal embedding equation (see for example equation (7.4) of \cite{Chen-Wang-Yau3}) implies that  $(X^0)^{(0)}$ solves the following equation 
\begin{equation}
\begin{split}
\frac{1}{2}\tilde \Delta (\tilde \Delta +2) (X^0)^{(0)} 
= & div_{S^2} \alpha_H^{(-1)}-  \sum_{i=1}^3 a^i \left[\tilde \Delta(\frac{\rho^{(-2)} \tilde X^i}{2}) +div_{S^2} \rho^{(-2)}(\tilde \nabla \tilde X^i)\right]\\
 = &  -\frac{1}{2} \tilde \Delta h^{(-2)} + \sum_{i=1}^3 \frac{a^i}{a^0} \left [div_{S^2} (h^{(-2)}\tilde \nabla \tilde X^i) + \tilde \Delta(\frac{h^{(-2)}\tilde X^i}{2})\right ].
\end{split}
\end{equation} 
This is solvable if and only if $(a^0,a^i)$ is the future directed unit timelike vector in the direction of the total energy momentum.

The total center of mass and angular momentum of $(M,g,k)$ is defined as follows.
\begin{definition}\label{total_conserved}
Suppose $ T_0({r})=A(r)(\frac{\partial}{\partial x^0})$ for a family of Lorentz transformation $A(r)$.  Define
\begin{equation}
C^i  =  \lim_{r \to \infty} E(\Sigma_r, X(r), T_0(r), A({r})(x^i\frac{\partial}{\partial x^0}+x^0\frac{\partial}{\partial x^i}))\\
\end{equation} to be the total center of mass integral and 
\begin{equation}
J_{i} =\lim_{r \to \infty} \epsilon_{ijk}E(\Sigma_r , X(r), T_0(r), A({r})(x^j\frac{\partial}{\partial x^k}-x^k\frac{\partial}{\partial x^j})) 
\end{equation} to be the total angular momentum,
where  $\Sigma_r$ are the coordinate spheres and $(X(r), T_0({r}))$ is the unique family of optimal isometric embeddings of $\Sigma_r$ such that $X(r)$ converges to a round sphere of radius $r$ in $\R^3$ when $r\rightarrow \infty$.
\end{definition}

In \cite{Chen-Wang-Yau3}, the authors studied the limit of quasi-local conserved quantities at the infinity of asymptotically flat initial data sets of order 1 and proved the following theorem concerning finiteness of total conserved quantities:
\begin{proposition}[Proposition7.1, \cite{Chen-Wang-Yau3}]
\label{proposition7.1}
The total center of mass and total angular momentum are finite if  
\begin{equation}\label{finite_condition}
\begin{split}
\int_{S^2} \tilde X^i \rho^{(-2)} dS^2 =0\text{ and } \int_{S^2}  \tilde X^i \left( \tilde{\epsilon}^{ab}\tilde{\nabla}_b(\alpha_H^{(-1)})_a \right) dS^2=0.
\end{split}
\end{equation}
\end{proposition}
Proposition 7.1 of \cite{Chen-Wang-Yau3} can be applied to the limit of quasi-local conserved quantities at infinity of  asymptotically hyperbolic initial data since $(\sigma, |H|, \alpha_H)$ has expansions of the same forms. 

In addition, the authors proved in \cite{Chen-Wang-Yau3} that equation  \eqref{finite_condition} follows from the vacuum constraint equation for asymptotically flat initial data sets of order 1. We prove in this section that, for coordinate spheres in asymptotically hyperbolic initial data sets, equation \eqref{finite_condition} does not follow from  the vacuum constraint equation. Instead, it holds only for the foliation with vanishing linear momentum.
\begin{proposition}\label{prop_finite_condition}
Let $\Sigma_r$ be a family of coordinate spheres in an asymptotically hyperbolic initial data sets. Then 
\begin{equation}
\begin{split}
\int_{S^2} \tilde X^i \rho^{(-2)} dS^2 =0\text{ and } \int_{S^2}  \tilde X^i \left( \tilde{\epsilon}^{ab}\tilde{\nabla}_b(\alpha_H^{(-1)})_a \right) dS^2=0
\end{split}
\end{equation}
holds if and only if the linear momentum vanishes. 
\end{proposition}
\begin{proof}
Recall from Lemma \ref{data_expansion}, we have
\begin{equation}
\begin{split}
h^{(-2)} =  - m \text{ and } (\alpha_H^{(-1)})_a =  \frac{1}{2}\partial_a m,
\end{split}
\end{equation} where $m$ is the mass aspect function. By the expansion of the induced metric on $\Sigma_r $ in Lemma \ref{data_expansion}, 
\[ h_0^{(-2) } =0 .  \]
It follows immediately that 
\[   \int_{S^2} \tilde X^i \left( \tilde{\epsilon}^{ab}\tilde{\nabla}_b(\alpha_H^{(-1)})_a \right) dS^2=0\]
holds and  
\[ \int_{S^2} \tilde X^i \rho^{(-2)} dS^2 = \frac{P^i_{AH}}{E_{AH}}. \]
\end{proof}
\begin{corollary}
For an asymptotically hyperbolic initial data set with timelike energy-momentum vector, the total angular momentum and center of mass associated with any foliation
with vanishing linear momentum are finite.
\end{corollary}
\begin{proof}
The corollary follows from combining Corollary \ref{cor_vanish_momentum}, Proposition \ref{proposition7.1}, and Proposition \ref{prop_finite_condition}.
\end{proof}
\section{Evaluating total center of mass and angular momentum.}
Let $(M,g,k)$ be an asymptotically hyperbolic initial data set. Suppose a foliation $\Sigma_r$ has vanishing linear momentum, we evaluate its  total center of mass and angular momentum. We first evaluate the total center of mass and angular momentum in terms of the expansion of $(\sigma,|H|,\alpha_H)$ and $(|H_0|,\alpha_{H_0} )$. Then we express them in terms of the expansion for $g$ and $k$. 

Let $(X(r), T_0(r))$  be the solution of the optimal embedding equation  on $\Sigma_r$. The solution takes a simpler form since the linear momentum vanishes.
In particular, we have
\begin{equation}  \label{eq_expansion_optimal}
\begin{split}
X^0(r) = & (X^0)^{(0)} + O(r^{-1})\\
X^i (r)= & r\tilde X^i+\frac{(X^i)^{(-1)}}{r}+ O(r^{-2})\\
T_0(r) = & (1,0,0,0) + \frac{T_0^{(-1)}}{r}+  O(r^{-2})
\end{split}
\end{equation}
where $(X^0)^{(0)}$ satisfies the equation  
\begin{equation}\label{optimal_eq} \frac{1}{2}\tilde \Delta (\tilde \Delta +2) (X^0)^{(0)} =-\frac{1}{2} \tilde \Delta h^{(-2)}.\end{equation} 
Let $T_0^{(-1)} = (0,b^1,b^2,b^3)$. The time function $\tau= - X(r) \cdot T_0(r)$ has the following expansion
\begin{equation}\label{tau_expansion} \tau=  (X^0)^{(0)} -  \sum_i b^i \tilde X^i  + O(r^{-1}).\end{equation}

\begin{proposition}
The total center of mass and angular momentum of an asymptotically hyperbolic initial data set are
\begin{equation}
\begin{split}
C^i =& \frac{1}{8 \pi} \int_{S^2}  \tilde X^i  (h_0^{(-3)}- h^{(-3)}) \ dS^2 \\
J^i =& \frac{-1}{8 \pi} \int_{S^2} \tilde X^i   \epsilon^{ab}  \tilde \nabla_b (\alpha_H^{(-2)})_a  dS^2. 
\end{split}
\end{equation}
\end{proposition}
\begin{proof}
Since $\tau=O(1)$,
\begin{equation}
\rho =  |H_0| - |H| + O(r^{-4}).
\end{equation}
Recall that,
\[ j_a = \rho {\nabla_a \tau }- \nabla_a [ \sinh^{-1} (\frac{\rho\Delta \tau }{|H_0||H|})]-(\alpha_{H_0})_a + (\alpha_{H})_a   \]
and 
\begin{equation}
\begin{split}
   j_a^{(-1)} =  & -(\alpha_{H_0})^{(-1)}_a + (\alpha_{H})^{(-1)}_a  \\
=  & \frac{1}{2} \left[-\tilde \nabla_a  (\tilde \Delta  +2)  (X^0)^{(0)} - \tilde \nabla h^{(-2)} \right]  =0,
\end{split}
\end{equation} by the optimal isometric embedding equation. 
The Killing vector field associated with the center of mass is  $K_{0i}=A({r})(x^i\frac{\partial}{\partial x^0}+x^0\frac{\partial}{\partial x^i})$
with $A(r) = Id + O(r^{-1})$. This implies that  $  (K_{0i}^\top)_a = O(r) $.
We also have
\[   \langle  K_{0i} , T_0(r) \rangle = -r \tilde  X^i + O(r^{-1}).\]
It follows that 
\begin{equation}
\begin{split}
    & E(\Sigma_r, X(r), T_0(r), A({r})(x^i\frac{\partial}{\partial x^0}+x^0\frac{\partial}{\partial x^i}))\\
= &  \frac{-1}{8 \pi}\int_{\Sigma_r}\left[ \langle  K_{0i} , T_0(r) \rangle \rho +( K_{0i}^\top)_a j^a \right] d\Sigma_r \\ 
= & \frac{1}{8 \pi} \int_{S^2} (h_0^{(-3)}- h^{(-3)}) \tilde X^i dS^2 +O(r^{-1}).
\end{split}
\end{equation}
For the angular momentum, the corresponding Killing vector field is  $K_{ij}=A({r})(x^i\frac{\partial}{\partial x^j}-x^j\frac{\partial}{\partial x^i})$. 
We have $ \langle  K_{ij} , T_0(r) \rangle =0$. Hence
\begin{equation}\label{angular}
\begin{split}
    & E(\Sigma_r, X(r), T_0(r), A({r})(x^i\frac{\partial}{\partial x^j}-x^j\frac{\partial}{\partial x^i}))\\
= & \frac{(-1)}{8\pi} \int_{\Sigma_r}
K_{ij}^\top\cdot \left (  \frac{\sqrt{(1+|\nabla \tau|^2)|H_0|^2 +(\Delta \tau)^2}\nabla \tau }{(1+|\nabla \tau|^2)} - \nabla[ \sinh^{-1} (\frac{\Delta \tau }{|H_0| \sqrt{1+|\nabla \tau|^2}})]-\alpha_{H_0}\right)d\Sigma_r  \\
& +  \frac{1}{8\pi} \int_{\Sigma_r}
K_{ij}^\top\cdot \left(  \frac{\sqrt{(1+|\nabla \tau|^2)|H|^2 +(\Delta \tau)^2}\nabla \tau }{(1+|\nabla \tau|^2)} - \nabla[ \sinh^{-1} (\frac{\Delta \tau }{|H| \sqrt{1+|\nabla \tau|^2}})] -\alpha_{H} \right )d\Sigma_r.
\end{split}
\end{equation}
We prove that the first integral is $O(r^{-1})$ by applying the conservation law to the image of the isometric embedding $X({r})$ of $\Sigma_r$, similar to the proof of  equation (8.1) in \cite{Chen-Wang-Yau3}.
Let $\mathcal C_r$ be the timelike cylinder in $\R^{3,1}$ obtained by translating $X(r)(\Sigma_r)$ along $T_0(r)$. Let $\hat \Sigma_r$ be the projection of $X(r)(\Sigma_r)$
onto the orthogonal complement of $T_0(r)$ and $\Omega_r$ be the portion of $\mathcal C_r$ between $X(r)(\Sigma_r)$ and $\hat \Sigma_r$.  Let $\gamma_{ij}$ and  $\Theta_{ij}$ be the metric and second fundamental form of $\mathcal C_r$ and  
$\pi_{ij}= \Theta_{ij} - tr _{\gamma} \Theta \gamma_{ij}$ be the conjugate momentum. Let $\nu$ be the outward unit normal of $\mathcal C_r$. From the expansion of $X(r)$, it is easy to see that 
\[  \langle K_{ij} , \nu \rangle = O(r^{-1}). \]
Let 
\[ K^{\mathcal C} =  K_{ij} -\langle K_{ij} , \nu \rangle \nu \]
be the tangential part of $K_{ij}$ to $\mathcal C_r$ and consider the vector field 
\[  Z= \pi( K^{\mathcal C}, \cdot). \]
We claim that $div_{\mathcal C} Z = O(r^{-3})$. Indeed, the conjugate momentum is divergence free since $\R^{3,1}$ is flat . We also have $K_{ij}$ is Killing in $\R^{3,1}$.  
As a result, we have 
\[  div_{\mathcal C} Z = \pi_{ij}  \Theta_{ij}  \langle K , \nu \rangle = O(r^{-3}). \]
Applying the divergence theorem to $Z$, we derive
\[ \int_{\Sigma_r } \pi (K^{\mathcal C}, n ) = \int_{\hat \Sigma_r}  \pi (K^{\mathcal C}, \hat n )  + \int_{\Omega_r}  div_{\mathcal C} Z  \]
where $n$ and $\hat n$ are the unit normal of $X(r)(\Sigma_r)$ and $\hat \Sigma_r$ in  $\mathcal C_r$.

By definition, $\int_{\Sigma_r } \pi (K^{\mathcal C}, n ) $ is the reference term in the angular momentum. Moreover, since $\hat \Sigma_r $ lies in a totally geodesics slice in $\R^{3,1}$
we have 
\[ \int_{\hat \Sigma_r}  \pi (K^{\mathcal C}, \hat n ) = 0.\]
Finally, since $ div_{\mathcal C} Z = (r^{-3})$, we have $\int_{\Omega_r}  div_{\mathcal C} Z = O(r^{-1})$.  It follows that the first integral in 
\eqref{angular} is of the order $O(r^{-1})$ as desired.

We proceed to study the second integral in \eqref{angular}.  We have
\[(K_{ij}^\top)_a =  r^2 (\tilde X^i \tilde \nabla_a \tilde X^j -\tilde X^j \tilde \nabla_a \tilde X^i)  + O(1)\]
and
\begin{equation*}
\begin{split}
    &  \frac{1}{8\pi} \int_{\Sigma_r}
K_{ij}^\top\cdot  \left(  \frac{\sqrt{(1+|\nabla \tau|^2)|H|^2 +(\Delta \tau)^2}\nabla \tau }{(1+|\nabla \tau|^2)} - \nabla[ \sinh^{-1} (\frac{\Delta \tau }{|H| \sqrt{1+|\nabla \tau|^2}})]-\alpha_{H}\right)d\Sigma_r\\
= & \frac{1}{8\pi}\int_{S^2} (\tilde X^i \tilde \nabla^a \tilde  X^j -\tilde X^j \tilde \nabla^a \tilde X^i)  \left (h^{(-2)}\tilde \nabla_a  (X^{(0)}_0  -  \sum_i b^i \tilde X^i ) - \frac{1}{2}\tilde  \nabla_a  (\tilde \Delta  X^{(0)}_0) -(\alpha_H^{(-2)})_a  \right ) dS^2+ O(r^{-1}).\\
\end{split}
\end{equation*}

Notice that \begin{equation}\label{cross_product} \tilde X^i \tilde \nabla^a \tilde X^j -\tilde X^j \tilde \nabla^a \tilde X^i=\tilde{\epsilon}^{ab}\epsilon^{ij}_{\,\,\,\,k} \tilde{\nabla}_b \tilde{X}^k\end{equation} is divergence 
free and thus
\[ \int_{S^2}  (\tilde X^i \tilde \nabla^a \tilde  X^j -\tilde X^j \tilde \nabla^a \tilde X^i) \tilde  \nabla_a  (\tilde \Delta  X^{(0)}_0) dS^2 =0. \]

We also claim that 
\begin{equation}\label{vanishing}\int_{S^2} ( \tilde X^i \tilde \nabla^a\tilde  X^j -\tilde X^j \tilde \nabla^a \tilde X^i)[h^{(-2)}\tilde \nabla_a  ( X^{(0)}_0  -  \sum_k b^k \tilde X^k) ]dS^2=0\end{equation}
Recall from \eqref{optimal_eq} that $ \tilde \Delta (\tilde \Delta +2)   (X^0)^{(0)} =-\tilde \Delta h^{(-2)}$, or  $(\tilde \Delta +2)   (X^0)^{(0)}$ and $-h^{(-2)}$
differ by a constant. 

In view of \eqref{cross_product}, \eqref{vanishing} is implied by 
\[ \int_{S^2} \tilde{\epsilon}^{ab}\tilde{\nabla}_b \tilde{X}^l  \tilde \nabla_a  (X^{(0)}_0-\sum_k b^k\tilde{X}^k) (\tilde \Delta +2)  X_0^{(0)}dS^2 =0\]
for $l=1,2,3$. 

Define the differential operator $L^l$ by $L^l(f)=\tilde{\epsilon}^{ab}\tilde{\nabla}_b \tilde{X}^l  \tilde \nabla_a f$ for any $C^1$ function on $S^2$. It suffices to show that 
\[\int_{S^2} L^l( X^{(0)}_0-\sum_k b^k \tilde{X}^k)(\tilde{\Delta}+2) X^{(0)}_0 dS^2=0.\]

$L^l$ corresponds to a rotational Killing field on $S^2$ and it is easy to check that the following properties hold for $L^l$: 
\[ L^l(fg)=L^l (f) g+ f L^l(g), [\tilde{\Delta}, L^l] f=0, \text{  and  } \int_{S^2} L^l(f) dS^2=0, \] for any $C^1$ functions $f$, $g$ on $S^2$. 

Integrating by parts and applying the above relations, we derive
\[\begin{split}& \int_{S^2} L^l( X^{(0)}_0-\sum_k b^k \tilde{X}^k)(\tilde{\Delta}+2) X^{(0)}_0 dS^2\\
&=\int_{S^2} L^l( (\tilde{\Delta}+2) (X^{(0)}_0-\sum_k b^k \tilde{X}^k)) X^{(0)}_0 dS^2\\
&=\int_{S^2} L^l( (\tilde{\Delta}+2) (X^{(0)}_0)) X^{(0)}_0 dS^2\\
&=\int_{S^2} L^l( \tilde{\Delta}X^{(0)}_0) X^{(0)}_0 dS^2.\end{split}\]

On the other hand, 
\[\int_{S^2} L^l( \tilde{\Delta}X^{(0)}_0) X^{(0)}_0 dS^2=\int_{S^2} \tilde{\Delta} (L^l( X^{(0)}_0)) X^{(0)}_0 dS^2
=\int_{S^2} (L^l( X^{(0)}_0)) \tilde{\Delta} X^{(0)}_0 dS^2.\]

Therefore,
\[\begin{split}&\int_{S^2} L^l( \tilde{\Delta}X^{(0)}_0) X^{(0)}_0 dS^2\\
&=\frac{1}{2}\int_{S^2} L^l( \tilde{\Delta}X^{(0)}_0) X^{(0)}_0 dS^2
+\frac{1}{2}\int_{S^2} (L^l( X^{(0)}_0)) \tilde{\Delta} X^{(0)}_0 dS^2\\
&=\frac{1}{2} \int_{S^2} L^l( X^{(0)}_0 \tilde{\Delta} X^{(0)}_0)=0\end{split}\]

This finishes the proof of the proposition.
\end{proof}
We need one more lemma before we evaluate the limit in terms of the data $(M,g,k)$.
\begin{lemma}\label{center_mass_2}
\[ \int_{S^2} \tilde X^i (h_0^{(-3)}-h^{(-3)}) dS^2 = -\int_{S^2} \tilde X^i \left[h^{(-3)}  +\frac{(h^{(-2)})^2}{4} \right]dS^2.\]
\end{lemma}
\begin{proof}
It suffices to prove \[\int_{S^2} \tilde X^i \left[h_0^{(-3)}+\frac{(h^{(-2)})^2}{4}\right] dS^2=0.\]
Recall that 
\[  |H_0|^2 =  \sum_i (\Delta X^i)^2 - (\Delta X^0)^2 \]
where
\begin{equation}
\begin{split}
\Delta X^0 = & \frac{ \tilde \Delta X^{(0)}_{0}}{r^2} + O(r^{-3})\\
\Delta X^i =  & r \Delta \tilde X^i +  \frac{\tilde \Delta (X^{i})^{(-1)} }{r^3}  + O(r^{-4})
\end{split}
\end{equation}
Furthermore, let $\Gamma_{ab}^c$ be the Christoffel symbols of the metric $\sigma_{ab}$. We have
\[ \Gamma_{ab}^c = \tilde   \Gamma_{ab}^c + \frac{( \Gamma^{(-2)})_{ab}^c}{r^2} + O(r^{-3}).\] 
This gives the following expansion for $\Delta \tilde X^i$
\[
\begin{split} 
\Delta \tilde X^i
= &  \frac{-2 \tilde X^i}{r^2}  +  \frac{ \tilde \sigma^{ab} g_{ab}^{(0)} \tilde X^i + \tilde \sigma^{ab}( \Gamma^{(-2)})_{ab}^c \tilde X^i_c  }{r^4} + O(r^{-5})  \\ 
=  &  \frac{-2 \tilde X^i}{r^2}  +  \frac{ \tilde \sigma^{ab} ( \Gamma^{(-2)})_{ab}^c \tilde X^i_c }{r^4} + O(r^{-5}) .
\end{split}
 \]
As a result, we have
\[ h_{0}^{(-3)} =  - \sum_k \tilde X^k  \tilde \Delta (X^{k})^{(-1)}  -\frac{1}{4}  (\tilde \Delta X^{(0)}_{0})^2. \]
Recall that  $({X}^i)^{(-1)}$ is the solution of the linearized isometric embedding equation
\begin{equation}\sum_{i=1}^3 \tilde \nabla _a \tilde{X}^i  \tilde \nabla _b ({X}^i)^{(-1)} + \tilde \nabla _b \tilde{X}^i  \tilde \nabla _a ({X}^i)^{(-1)} =  \tilde \nabla _a ({X}^0)^{(0)}  \tilde \nabla _b ({X}^0)^{(0)} + g^{(0)}_{ab} .\end{equation} 
Let $T_{ab}= \tilde \nabla _a ({X}^0)^{(0)}  \tilde \nabla _b ({X}^0)^{(0)}+ g^{(0)}_{ab}$. We use the ansatz in Nirenberg's  paper \cite{Nirenberg} for solving the isometric embedding of surface with positive Gauss curvature into $\R^3$. Write
\[ \tilde \nabla_a ({X}^j)^{(-1)}= P_a \tilde X^j + (\frac{T_{ab}}{2} + F \epsilon_{ab}) \tilde \sigma^{bc} \tilde X^j_c  \]
where $P_a$ is a vector field,  $F$ is a function, and $ \epsilon_{ab}$ is the area form on $S^2$.  $P_a$ and $F$ as considered to be the new unknowns that satisfy the compatibility condition. We have
\begin{equation} \label{Nirenberg_equ}
 \tilde \nabla_d \tilde \nabla_a X^{(-1)} = (\tilde \nabla_d P_a) \tilde X  + P_a \tilde X_d +  (\frac{\tilde \nabla _d T_{ab}}{2} + \tilde \nabla_d F \epsilon_{ab}) \tilde \sigma^{bc} \tilde X_c -  (\frac{T_{ad}}{2} + F \epsilon_{ad}) \tilde X .
\end{equation}
Taking the trace of equation (\ref{Nirenberg_equ}) and multiplying by $-\tilde{X}$,
\[\sum_k  - \tilde X^ k \tilde \Delta  ({X}^k)^{(-1)} = - \tilde \nabla^a P_a + \frac{1}{2}\tilde \sigma^{ab}T_{ab}=- \tilde \nabla^a P_a +\frac{1}{2}|\tilde \nabla  ({X}^0)^{(0)} |^2.\]

Multiplying by $\tilde{X}^i$ and integrating, we obtain
\begin{equation}\label{integral_dth_0}-\int \tilde{X}^i  \sum_k \tilde X^k  \tilde \Delta (X^{k})^{(-1)}dS^2 =\frac{1}{2}\int \tilde{X}^i|\tilde \nabla  ({X}^0)^{(0)} |^2 -\int \tilde{X}^i\tilde{\nabla}^a P_a dS^2.\end{equation}
We  use the compatibility condition, $\epsilon^{ad}\tilde \nabla_d \tilde \nabla_a X^{(-1)} =0 $ and \eqref{Nirenberg_equ}, to obtain equations for $P_a$ and $F$:
\[
\begin{split}
\epsilon^{ad}( \tilde \nabla_d P_a  - \frac{T_{ad}}{2} - F \epsilon_{ad})& =0,\\
\epsilon^{ad}(P_a \tilde \sigma_{cd} +\frac{1}{2}\tilde \nabla _dT_{ac} +\tilde \nabla _d  F \epsilon_{ac}) & =0.
\end{split}
\]
$ P_a$ can be solved  from the second equation in terms of $F$ and $T_{ab}$. The equation can be written as 
\begin{equation} \label{equ_P}
\epsilon_{ac}P_b \tilde \sigma^{ab} +\frac{1}{2} \epsilon^{ad} \tilde\nabla _dT_{ac}+ \tilde \nabla _c F   =0.
\end{equation}
Multiplying  the equation by $\epsilon^{ce}$, we have
\[ P^e = \frac{1}{2} \epsilon^{ce} \epsilon ^{ad} \tilde\nabla _dT_{ac} + \epsilon^{ce} \tilde \nabla _c F.    \]
Taking the divergence and integrating against $\tilde{X}^i$,
\[
\begin{split}
\int_{S^2} \tilde X^i \tilde \nabla ^e P_e  dS^2 &= \int_{S^2} \tilde X_i \tilde \nabla _e(\frac{1}{2} \epsilon^{ce} \epsilon ^{ad} \tilde\nabla _dT_{ac} + \epsilon^{ce} \tilde \nabla _c F) dS^2 \\
& =  \frac{1}{2}\int_{S^2}\tilde \nabla _e \tilde\nabla _d \tilde X^i  \epsilon^{ce} \epsilon ^{ad}T_{ac}dS^2\\
& =  -\frac{1}{2} \int_{S^2} \tilde X^i \tilde \sigma_{de}  \epsilon^{ce} \epsilon ^{ad}T_{ac}dS^2\\
& = - \frac{1}{2} \int_{S^2} \tilde X^i \tilde \sigma^{ab}T_{ab} dS^2.
\end{split}
\]
As a result, we have
\[
-\int_{S^2} \tilde X^i   \sum_k \tilde X^k  \tilde \Delta (X^{k})^{(-1)} = \int_{S^2} \tilde X^i  |\tilde \nabla (X^0)^{(0)}  |^2
\]
and
\begin{equation}
\begin{split}
   & \int_{S^2} \tilde X^i \left (  h_0^{(-3)}+\frac{(h^{(-2)})^2}{4}\right) dS^2  \\
= &  \int_{S^2} \tilde X^i  \left ( |\tilde \nabla  (X^0)^{(0)} |^2  -\frac{1}{4}  (\tilde \Delta X^{(0)}_{0})^2  +\frac{(h^{(-2)})^2}{4}\right)dS^2  \\
= &  \int_{S^2} \tilde X^i  \left (|\tilde \nabla  (X^0)^{(0)}  |^2  -\frac{1}{4}  (\tilde \Delta  (X^0)^{(0)})^2  +\frac{[(\tilde \Delta + 2)  (X^0)^{(0)} ]^2}{4} \right)dS^2  \\
= &  \int_{S^2} \tilde X^i  \left (|\tilde \nabla  (X^0)^{(0)}  |^2    + (  (X^0)^{(0)}) ^2 +  X^{(0)}_{0} \tilde  \Delta  (X^0)^{(0)}\right) dS^2 \\
= &  \int_{S^2} \tilde X^i  \left ( (  (X^0)^{(0)}) ^2 + \frac{\tilde \Delta}{2}  ( (X^0)^{(0)})^2 \right) dS^2 =   0.
\end{split}
\end{equation}
\end{proof}
We are now ready to express the total angular momentum and center of mass in terms of the expansion of the metric $g$ and second fundamental form $k$. The calculation is straight-forward and is similar to Lemma 3.1. We simply need to expand one more order. 
\begin{theorem}\label{thm_conserved}
Given an asymptotically hyperbolic initial data set $(M,g,k)$, for the foliation with vanishing linear momentum, the center of mass $C^i$ and angular momentum $Ji$ of $(M,g,k)$ are 
\begin{equation}
\begin{split}
C^i =& \frac{1}{8 \pi} \int_{S^2} \tilde X^i (2 tr_{S^2} g_{ab}^{(-2)} +g_{rr}^{(-6)}+ \tilde \nabla^a g_{ra}^{(-3)} +  tr_{S^2} p_{ab}^{(-2)}) dS^2 \\
J^i =& \frac{1}{8 \pi} \int_{S^2} \tilde X^i  \tilde\epsilon^{ab}  \tilde \nabla_b (g_{ra}^{(-3)}+p_{ra}^{(-3)}) dS^2 .
\end{split}
\end{equation}
\end{theorem}
\begin{proof}
We start with the angular momentum. Recall that 
\[ 
J^i   = \frac{-1}{8 \pi} \int_{S^2} \tilde X^i   \epsilon^{ab}  \tilde \nabla_b (\alpha_H^{(-2)})_a  dS^2.
 \]
where
\[ 
\begin{split}
 ( \alpha_H)_a = & - k(e_3, \partial_a) + \nabla_a \theta \\
 k(e_3, \partial_a)        =& \frac{g_{ra}^{(-3)}+p_{ra}^{(-3)}}{r^2} + O(r^{-3}) . 
\end{split}
\]
Hence,
\[ 
\begin{split}
  J^i  
 =&\frac{1}{8 \pi} \int_{S^2} \tilde X^i  \tilde\epsilon^{ab} \left [ \tilde \nabla_b (g_{ra}^{(-3)}+p_{ra}^{(-3)})   + \tilde \nabla_b \tilde \nabla _a \theta \right]dS^2 .\\
=  &\frac{1}{8 \pi} \int_{S^2} \tilde X^i  \tilde\epsilon^{ab}  \tilde \nabla_b (g_{ra}^{(-3)}+p_{ra}^{(-3)}) dS^2 .
\end{split}
\]
For the center of mass, we start with $\langle H, e_4 \rangle$. 
\[  
\begin{split}
   -\langle H, e_4 \rangle
= &   \sigma^{ab}( g_{ab}  + p_{ab} ) \\
=& 2 + \frac{ tr_{S^2} p^{(-1)}_{ab} }{r^3}+ \frac{ tr_{S^2} p^{(-2)}_{ab} }{r^4} + O(r^{-5}).
\end{split}
\]
On the other hand, we compute
\[   e_3 =  \frac{1}{\sqrt{ g_{rr} -|V|_{\sigma}^2 }} (\frac{\partial}{\partial r}  + V^a\frac{\partial}{\partial u^a}  )   \] 
where  $V_a = -g_{ra}$. The second fundamental form of $\Sigma_r$ in the direction of $e_3$ is 
\[  
\begin{split}
        \langle \nabla_{\frac{\partial}{\partial u^a}  }   e_3 , \frac{\partial}{\partial u^b}  \rangle = & \sqrt{r^2+1 - \frac{ g^{(-5)}_{rr}}{r}- \frac{ g^{(-6)}_{rr}}{r^2}  }  \langle \nabla_{\frac{\partial}{\partial u^a}  }  \frac{\partial}{\partial r} + V^c \frac{\partial}{\partial u^c}, \frac{\partial}{\partial u^b}  \rangle  + O(r^{-3}) \\
=&  \sqrt{r^2+1 - \frac{ g^{(-5)}_{rr}}{r}- \frac{ g^{(-6)}_{rr}}{r^2} } \left[r \tilde \sigma_{ab} - \frac{g^{(-1)}_{ab}}{2 r^2}- \frac{(g^{(-2)}_{ab}+ \tilde \nabla_a  g^{(-3)}_{br} )  }{r^3} \right]+ O(r^{-3})   .
\end{split}
\]
Hence, 
\[  
\begin{split}
    &   - \langle H, e_3 \rangle   \\
= &  \sigma^{ab}   \sqrt{r^2+1 - \frac{ g^{(-5)}_{rr}}{r}- \frac{ g^{(-6)}_{rr}}{r^2}} \left[r \tilde \sigma_{ab} - \frac{g^{(-1)}_{ab}}{2 r^2}- \frac{(g^{(-2)}_{ab}+ \tilde \nabla_a  g^{(-3)}_{br})}{ r^3}\right]  + O(r^{-5})  \\
= &  r(1+ \frac{1}{2r^2} - \frac{ g^{(-5)}_{rr}}{2r^3} - \frac{g^{(-6)}_{rr}+ \frac{1}{4}}{2r^3})(r \tilde \sigma_{ab} - \frac{g^{(-1)}_{ab}}{2 r^2}- \frac{g^{(-2)}_{ab}+ \tilde \nabla_a  g^{(-3)}_{br}}{ r^3}) (\frac{ \tilde \sigma^{ab}}{r^2} - \frac{g^{(-1)})^{ab}}{r^5}- \frac{(g^{(-2)})^{ab}}{r^6} ) + O(r^{-5}) \\
= & 2+ \frac{1}{r^2} - \frac{ g^{(-5)}_{rr} + \frac{3}{2} tr_{S^2} g^{(-1)}_{ab} }{r^3} - \frac{g^{(-6)}_{rr}+\frac{1}{4} + 2 tr_{S^2} g^{(-1)}_{ab}+ \tilde \nabla^a  g^{(-3)}_{ar}}{r^4}+ O(r^{-5}).
\end{split}
\]
We have
\[  
\begin{split}
|H|^2 = &   \langle H, e_3 \rangle ^2 - \langle H, e_4 \rangle ^2 \\ 
=& \frac{4}{r^2} + \frac{4h^{(-2)}}{r^3}-\frac{4(g^{(-6)}_{rr} + 2 tr_{S^2} g^{(-1)}_{ab}+ \tilde \nabla^a  g^{(-3)}_{ar}+ tr_{S^2} p^{(-2)}_{ab})}{r^4}  +O(r^{-5}) .
\end{split}
\]
Recall 
\[ |H| = \frac{2}{r} + \frac{h^{(-2)}}{r^2}+ \frac{h^{(-3)}}{r^3} + O(r^{-4}). \]
Matching the coefficients, 
\[  h^{(-3)}  +\frac{(h^{(-2)})^2}{4}   = -(g^{(-6)}_{rr} + 2 tr_{S^2} g^{(-1)}_{ab}+ \tilde \nabla^a  g^{(-3)}_{ar}+ tr_{S^2} p^{(-2)}_{ab}).  \]
The formula for the center of mass follows from Lemma \ref{center_mass_2}.
\end{proof}


\begin{thebibliography}{99}  

\bibitem{Andersson-Cai-Galloway}
L. Andersson, M. Cai\ and\ G. J. Galloway, \textit{Rigidity and positivity of mass for asymptotically hyperbolic manifolds}, Ann. Henri Poincar\'e {\bf 9} (2008), no.~1, 1--33.
\bibitem{Arnowitt-Deser-Misner} R. Arnowitt, S. Deser\ and\ C. W. Misner, The dynamics of general relativity, in {\it Gravitation: An introduction to current research}, 227--265, Wiley, New York.
\bibitem{Ashtekar-Hansen} A. Ashtekar and R. O. Hansen, \textit{A Unified Treatment of Spatial and Null Infinity in General Relativity: Universal Structure, Asymptotic Symmetries and conserved Quantities at Spatial Infinity},  J. Math. Phys, {\bf 19} (1978) , 1542-1566.
\bibitem{Bartnik} R. Bartnik, \textit{The mass of an symptomatically flat manifold,} Commun. Pure Appl. Math. 39, 661--693 (1986).
\bibitem{Beig-Omurchadha} R. Beig\ and\ N. \'O Murchadha, \textit{The Poincar\'e group as the symmetry group of canonical general relativity}, Ann. Physics {\bf 174} (1987), no.~2, 463--498. 
\bibitem{Bondi-Burg-Metzner} H. Bondi, M. G. J. van der Burg and A. W. K. Metzner.  \textit{ Gravitational Waves in General Relativity. VII. Waves from Axi-Symmetric Isolated Systems}, Proc. Roy. Soc. A. \textbf{269} (1962) 21--52
\bibitem{Chen-Wang-Yau1} P.-N. Chen, M.-T. Wang, and S.-T. Yau, \textit {Evaluating quasilocal energy and solving optimal embedding equation at null infinity,} Comm. Math. Phys. \textbf{308} (2011), no.3, 845--863
\bibitem{Chen-Wang-Yau2}  P.-N. Chen, M.-T. Wang, and S.-T. Yau, \textit{Minimizing properties of critical points of quasi-local energy}, Comm. Math. Phys. \textbf{329} (2014), no.3, 919--935
\bibitem{Chen-Wang-Yau3}  P.-N. Chen, M.-T. Wang, and S.-T. Yau, \textit {Conserved quantities in general relativity: from the quasi-local level to spatial infinity}, arXiv:1312.0985
\bibitem{Christodoulou2} D. Christodoulou, \textit{Nonlinear nature of gravitation and gravitational-wave experiments}, Phys. Rev. Lett. {\bf 67} (1991), no.~12, 1486--1489. 
\bibitem{Christodoulou}  D. Christodoulou, {\it Mathematical problems of general relativity. I}, Zurich Lectures in Advanced Mathematics, European Mathematical Society (EMS), Z\"urich, 2008. 
\bibitem{sta} D. Christodoulou and S. Klainerman. \textit{The global nonlinear stability of the Minkowski space}, Princeton Math. Series \textbf{41}. Princeton University Press. Princeton. NJ. (1993). 
\bibitem{Chrusciel} P. T. Chru\'sciel, \textit{On the invariant mass conjecture in general relativity}, Comm. Math. Phys. {\bf 120} (1988), no.~2, 233--248.
\bibitem{Chrusciel-Herzlich} P. T. Chru\'sciel\ and\ M. Herzlich, \textit{The mass of asymptotically hyperbolic Riemannian manifolds}, Pacific J. Math. {\bf 212} (2003), no.~2, 231--264. 

\bibitem{Chrusciel-Jezuerski}P. T. Chru\'sciel, J. Jezierski\ and\ S. \L\c eski, \textit{The Trautman-Bondi mass of hyperboloidal initial data sets}, Adv. Theor. Math. Phys. {\bf 8} (2004), no.~1, 83--139.
\bibitem{Fan-Shi-Tam} X.-Q. Fan, Y. Shi\ and\ L.-F. Tam, \textit{Large-sphere and small-sphere limits of the Brown-York mass}, Comm. Anal. Geom. {\bf 17} (2009), no.~1, 37--72. 
\bibitem{Huang}L.-H. Huang, \textit{On the center of mass of isolated systems with general asymptotics}, Classical Quantum Gravity {\bf 26} (2009), no.~1, 015012, 25 pp.
\bibitem{Huisken-Yau} G. Huisken\ and\ S.-T. Yau, \textit{Definition of center of mass for isolated physical systems and unique foliations by stable spheres with constant mean curvature}, Invent. Math. {\bf 124} (1996), no.~1-3, 281--311. 
\bibitem{Kwong-Tam} K.-K. Kwong\ and\ L.-F. Tam, \textit{Limit of quasilocal mass integrals in asymptotically hyperbolic manifolds}, Proc. Amer. Math. Soc. {\bf 141} (2013), no.~1, 313--324.
\bibitem{Liu-Yau} C.-C. M. Liu\ and\ S.-T. Yau, Positivity of quasi-local mass. II, J. Amer. Math. Soc. {\bf 19} (2006), no.~1, 181--204
\bibitem{Min-Oo}   M. Min-Oo, \textit{Scalar curvature rigidity of asymptotically hyperbolic spin manifolds}, Math. Ann. {\bf 285} (1989), no.~4, 527--539. 
\bibitem{Nirenberg} L. Nirenberg, \textit{The Weyl and Minkowski problems in differential geometry in the large}, Comm. Pure Appl. Math. {\bf 6} (1953), 337--394.
\bibitem{Regge-Teitelboim}T. Regge\ and\ C. Teitelboim, \textit{Role of surface integrals in the Hamiltonian formulation of general relativity}, Ann. Physics {\bf 88} (1974), 286--318
\bibitem{Rizzi} A. Rizzi, \textit{Angular momentum in general relativity: a new definition}, Phys. Rev. Lett. {\bf 81} (1998), no.~6, 1150--1153. 
\bibitem{Sachs} R. K. Sachs, \textit{Gravitational waves in general relativity. VIII. Waves in asymptotically flat space-time. Proc.} Roy. Soc. Ser. A \textbf{270} (1962) 103--126.
\bibitem{Schoen-Yau} R. Schoen and S.-T. Yau, \textit{Proof of the positive mass theorem. II.}  Comm. Math. Phys. \textbf{79} (1981), no. 2, 231--260.
\bibitem{Schoen-Yau2}R. Schoen\ and\ S. T. Yau, \textit{Proof that the Bondi mass is positive}, Phys. Rev. Lett. {\bf 48} (1982), no.~6, 369--371.
\bibitem{Sakovich} A. Sakovich, \textit{A study of asymptotically hyperbolic manifolds in mathematical relativity}, Ph. D. Thesis, (2012), KTH,  Stockholm,  Sweden.
\bibitem{Trautman} A. Trautman, \textit{Radiation and boundary conditions in the theory of gravitation}, Bull. Acad. Polon. Sci. S\'er. Sci. Math. Astr. Phys. {\bf 6} (1958), 407--412. 
\bibitem{Wang-Yau1} M.-T. Wang\ and\ S.-T. Yau, \textit{A generalization of Liu-Yau's quasi-local mass}, Comm. Anal. Geom. {\bf 15} (2007), no.~2, 249--282.
\bibitem{Wang-Yau2} M.-T. Wang and S.-T. Yau, \textit{Isometric embeddings into the Minkowski space and new quasi-local mass.} Comm. Math. Phys. \textbf{288} (2009), no. 3, 919-942.
\bibitem{Wang-Yau3} M.-T. Wang and S.-T. Yau, \textit{Limit of quasilocal mass at spatial infinity.} Comm. Math. Phys. \textbf{296} (2010), no.1, 271-283. arXiv:0906.0200v2.
\bibitem{Wang} X. Wang, \textit{The mass of asymptotically hyperbolic manifolds}, J. Differential Geom. {\bf 57} (2001), no.~2, 273--299.
\bibitem{Witten}E. Witten, \textit{A new proof of the positive energy theorem}, Comm. Math. Phys. {\bf 80} (1981), no.~3, 381--402.
\bibitem{Zhang}X. Zhang, \textit{Strongly asymptotically hyperbolic Spin manifolds}, Math. Res. Lett. {\bf 7} (2000), no.~5-6, 719--727. 
\bibitem{Zhang2}X. Zhang, \textit{A definition of total energy-momenta and the positive mass theorem on asymptotically hyperbolic 3-manifolds. I}, Comm. Math. Phys. {\bf 249} (2004), no.~3, 529--548.
\bibitem{Zipser}N. Zipser, Part II: \textit{Solutions of the Einstein-Maxwell equations}, in {\it Extensions of the stability theorem of the Minkowski space in general relativity}, 297--491, AMS/IP Stud. Adv. Math., 45, Amer. Math. Soc., Providence, RI.
\end{thebibliography}
\end{document}